\documentclass[11pt,leqno]{amsart}

\usepackage[utf8x]{inputenc}
\usepackage[T1]{fontenc}
\usepackage[english]{babel}
\usepackage{newunicodechar}
\newunicodechar{̆}{\u{}}
\usepackage{enumitem}
\usepackage[dvipsnames]{xcolor}
\usepackage{amsthm}
\usepackage{amsmath,amssymb}
\usepackage{mathdots}

\usepackage{tikz}
\usepackage{tikz-cd}
\usetikzlibrary{arrows, positioning, calc, patterns}

\usepackage{bbm}
 \usepackage{todonotes}

\newtheorem{theorem}{Theorem}[section]
\newtheorem{lemma}[theorem]{Lemma}
\newtheorem{proposition}[theorem]{Proposition}
\newtheorem{corollary}[theorem]{Corollary}

\theoremstyle{definition}
\newtheorem{definition}[theorem]{Definition}

\newtheorem{problem}[theorem]{Problem}

\theoremstyle{remark}

\numberwithin{equation}{section}

\DeclareMathOperator{\Z}{\mathbb Z}
\DeclareMathOperator{\Q}{\mathbb Q}

\DeclareMathOperator{\N}{\mathbb N}
\DeclareMathOperator{\Stab}{Stab}

\DeclareMathOperator{\diag}{diag}

\newcommand{\F}{\mathbb{F}}

\DeclareMathOperator{\Irr}{Irr}
\DeclareMathOperator{\Ind}{Ind}
\DeclareMathOperator{\Inf}{Inf}
\DeclareMathOperator{\Res}{Res}
\DeclareMathOperator{\Ext}{Ext}
\DeclareMathOperator{\idrep}{\mathbbm{1}}

\DeclareMathOperator{\p}{\mathfrak{p}}
\DeclareMathOperator{\Op}{\mathcal{O}}

\DeclareMathOperator{\Pfaff}{Pfaff}

\DeclareMathOperator{\Sym}{Sym}
\DeclareMathOperator{\SL}{SL}

\DeclareMathOperator{\GL}{GL}
\DeclareMathOperator{\GA}{GA}
\DeclareMathOperator{\Com}{Com}

\usepackage[refpage,intoc,notocbasic]{nomencl}
\usepackage{textcase} 
\usepackage{xpatch}

\usepackage[pagebackref, colorlinks=true, linkcolor=MidnightBlue, citecolor=OliveGreen]{hyperref}
\usepackage[capitalise,nameinlink]{cleveref}

\makeatletter
\@namedef{subjclassname@2020}{\textup{2020} Mathematics Subject Classification}
\makeatother

\renewcommand*{\backref}[1]{}
\renewcommand*{\backrefalt}[4]{%
  \ifcase #1 %
    No citations.
  \or
    (cited on page~#2).%
  \else
    (cited on pages~#2).%
  \fi%
}

\title[Zeta functions of semi-direct products]{Representation zeta functions\\ of split extensions of~$\SL_2^m(\Op)$}

\author[J. M. Petschick]{J. Moritz Petschick}
\address{J. Moritz Petschick: Fakult\"at für Mathematik, Universität Bielefeld, D-33501 Bielefeld, Germany}
\email{jpetschick@math.uni-bielefeld.de}

\author[M. Piccolo]{Margherita Piccolo} 
\address{Margherita Piccolo:
  Mathematisches Institut, FernUniversit\"at in Hagen, 58097 
  Hagen, Germany} \email{margherita.piccolo@fernuni-hagen.de}

\thanks{
  The first author was supported by the Deutsche Forschungsgemeinschaft (DFG, German Research Foundation) – 380258175 and SFB-TRR 358/1 2023 – 491392403.
  The second author was partially funded by the research training group GRK2240: Algebro-geometric Methods in Algebra, Arithmetic and Topology of the DFG
  and is a member of GNSAGA (INdAM) and kindly acknowledges their support.
}


\date{\today}

\begin{document}

\begin{abstract}
    We consider the representation growth of split extensions of~$\operatorname{SL}_2^m(\Op)$. We prove that the corresponding representation zeta functions factor as a product of the representation zeta function of~$\operatorname{SL}_2^m(\Op)$ and the relative representation zeta function associated to the extension. We make use of our result by computing the zeta functions for two infinite families of split extensions of~$\SL_2^m(\Op)$ explicitly. Along the way, we compute the representation zeta functions of a large class of subgroups of~$\SL_2^m(\Op)$.
\end{abstract}

\maketitle

\section{Introduction}

Let $G$ be a topological group, and let $r_n(G)$ be the cardinality of the set of isomorphism classes of $n$-dimensional continuous complex representations of $G$. The group $G$ is said to be \emph{representation rigid}, if $r_n(G)$ is finite for every~$n$. Within the class of finitely generated profinite groups, representation rigidity is characterised by the property usually called `FAb', which a group satisfies if every of its open subgroups has finite abelianisation; see~\cite[Proposition~2]{BLMM02} for further details. If the group $G$ has \emph{polynomial representation growth}, i.e.\ if the function $N \mapsto R_N(G) = \sum_{n=1}^N r_n(G)$ polynomially bounded, one defines the \emph{representation zeta function of $G$} as the Dirichlet generating function
\[
    \zeta_G(s) = \sum_{n = 1}^\infty r_n(G)n^{-s}
\]
defined on a left half-plane, such that its abscissa of convergence~$\alpha(G)$ coincides with the polynomial degree of growth of the sequence~$(R_N(G))_{N\in \N}$. These zeta functions and variations have been studied for a number of different classes of groups, see e.g.\ \cite{AKOV13,Klo13,LL08,SV14,OPS25,Wit91} and the references therein.

The representation zeta function of a FAb compact $p$-adic analytic groups are `almost' rational functions, as Jaikin-Zapirain~\cite{Jai05} established, for odd primes, that there exist positive integers $n_1, \dots, n_r$ and rational functions~$f_1,\dots,f_r\in \mathbb{Q}(t)$ such that $\zeta_G(s) =\sum_{i=1}^r n_i^{-s}f_i(p^{-s})$. This statement has been extended to the prime $2$ by Stasinski and Zordan~\cite{SZ20}. Recently, Kionke and Klopsch~\cite{KK19} considered the more general problem of counting not \emph{all} irreducible representations of a group $G$, but all irreducible constituents in a given representation $\varphi$ of $G$, counted by their multiplicity; giving rise to the \emph{zeta function associated to $\varphi$}. Under some natural restrictions on $\varphi$, much of the theory on representation zeta functions naturally generalises to this more flexible setting. We make heavy use of their results throughout the article.

Although the underlying techniques for the results above—based on the Kirillov orbit method, $p$-adic integration and relative Clifford theory—give a `recipe' for the computation of the representation zeta function, it is challenging obtain explicit formulae, which are known in just a handful of cases, all of them dealing with groups of small dimension and trivial soluble radical; see \cite{AKOV12, AKOV13, Jai05, Zor15}. Even the determination of the abscissae of convergence remains an open problem in general, though some bounds are known, cf.~\cite{AA16, AKOV12, Bud21, LL08}.

In contrast to previous examples, we consider compact $p$-adic analytic groups of arbitrary large dimension whose soluble radical is infinite. Concretely, we consider split extensions of principal congruence subgroups of the special linear group of degree~$2$ over a compact discrete valuation ring of characteristic $0$. We state our main result, deferring the precise definitions of some of the technical terms involved.

\begin{theorem}\label{thm:main}
    Let $\Op$ be a compact discrete valuation ring of characteristic $0$ and residue characteristic $p$. Let $\mathfrak{k}$ be a $3$-dimensional simple $\Op$-Lie lattice and let $\mathfrak{h}$ be an open $\Op$-Lie sublattice of $\mathfrak{k}$. Let $m$ be soundly permissible, $H = \exp\mathfrak{h}$, and let~$\sigma \colon H \to \GL_n^m(\Op)$ be a faithful finite-dimensional $\Op$\nobreakdash-representation of~$H$. Assume furthermore that the semi-direct product~$G = H \ltimes_{\sigma} \Op^n$ is FAb.
    Then the representation zeta function of $G$ satisfies
    \[
    	\zeta_{G} (s) = \zeta_H(s)\cdot\zeta^{G}_{H}(s-1),
    \]
    where~$\zeta^{G}_{H}(s)$ is the zeta function associated to the representation~$\Ind_{H}^G(\idrep_H)$.
\end{theorem}

This product decomposition enables us to compute explicit formulae for infinitely many split extensions of potent subgroups of $\SL_2(\Op)$.  
However, our examples show that such decomposition cannot always be expected in a more general setting, indicating that this behaviour is not universal.

Representation zeta functions of compact~$p$-adic analytic groups and of complex Lie groups appear as local factors in an Euler product decomposition of the representation zeta functions associated to ($S$-)arithmetic groups with the congruence subgroup property, see~\cite{LL08}. 
Special linear groups~$\SL_2({\scriptstyle\mathcal{O}})$ over the ring of integers of an algebraic number field~$\mathcal{K}$ do not have the congruence subgroup property; however, the congruence subgroup property holds as soon as we consider the $S$-arithmetic group~$\SL_2({\scriptstyle\mathcal{O}}_S)$ for a non-empty finite subset $S$ of the set of valuations of the number field~$\mathcal{K}$ with ring of integers~$\scriptstyle\mathcal{O}$, containing the Archimedean valuations and at least one non-Archimedean valuation, cf.\ \cite[$\S$ 9.5]{PR94}. In particular, for such $S$, we have
\[
    \zeta_{\SL_2(\mathbb{Z}_S)}(s)=\zeta_{\SL_2(\mathbb{C})}(s)\times \prod_{p\notin S}\zeta_{\SL_2(\mathbb{Z}_p)}(s).
\]
The group $\SL_2(\mathbb{C})$ has a unique irreducible rational representation of each degree, its representation zeta function (in this context also called the Witten zeta function) coincides with the Riemann zeta function. Indeed, every irreducible rational representation of~$\SL_2(\mathbb{C})$ is isomorphic to a symmetric power $\Sym^n(\mathbb{C}^2)$ of the standard representation; see, e.g.~\cite[Chapter~11]{FH91}. Inspired by the complex Lie group case, we consider the $2$- and $3$-dimensional representations of $\SL_2^m(\Op)$, where $\Op$ is a compact discrete valuation ring of characteristic~$0$, given by the natural action on~$\Op^2$, and by its symmetric square, respectively.

The integral analogues of split extensions of $\SL_2(\mathbb{Z})$ are of considerable interest themselves. The (integral) special affine group~$\SL_2(\mathbb{Z})\ltimes \mathbb{Z}^2$ served as a key ingredient in Kazhdan's proof establishing property~(T) for $\SL_3(K)$ for a local field $K$, cf.\ \cite{Kaz67,Bur91,dlHV89}. 
Recently, Zhang \cite{Zha24} proved that for a very general class of rings including global number rings, all groups of the form $\SL_2(R)\ltimes R^n$ have relative property (T), where the action of $\SL_2(R)$ on $R^n$ is defined via an irreducible representation of $\SL_2(\mathbb{C})$. Split extensions of this kind have also appeared in the construction of expander graphs, see \cite{Mar73,Lub94}.

One aim of our work is to produce explicit examples of $p$-adic analytic groups with arbitrary large degree of representation growth. To achieve this, we consider not only the natural representation of $\SL_2(\Op)$ and its symmetric square, but direct powers of these representations.

\begin{theorem}\label{thm:diagonal on natural module}
    Let~$\Op$ be a compact discrete valuation ring of characteristic~$0$, residue characteristic an odd prime~$p$, and residue field cardinality~$q$.
	For all $m \in \N$ that are soundly permissible for $\mathfrak{sl}_2(\Op)$ and all $n \in \N_+$, the representation zeta function of the semi-direct product of the principal congruence subgroup $\SL_2^m(\Op)$ and $\Op^{2n}$ with respect to the $n$\textsuperscript{th} power of the natural action is
	\[
	    \zeta_{\SL_2^m(\Op) \ltimes \Op^{2n}}(s) = \zeta_{\SL_2^m(\Op)}(s)\cdot q^{2nm}\frac{(1+q^{-s})f(q,q^{-s})}{(1-q^{n+1-2s})(1-q^{2n-3s})},
	\]
	where
	\[
		f(q,t) = 1-t+(q^{n-1}-1)qt^2+(1+t)q^{n+1}t^3.
	\]
\end{theorem}

The representation zeta function $\zeta_{\SL_2^m(\Op)}(s)$ is explicitly computed in~\cite[Theorem~1.2]{AKOV10}.
Note that the abscissae of convergence of the groups in \cref{thm:diagonal on natural module} depend on which of the two uniformly varying factors in the denominator dominates, and this varies with $n$. For small values ($n \leqslant 3$), the abscissa equals $\frac{n+1}{2}$, yielding the specific values $1$, $3/2$, and $2$ for $n = 1, 2, 3$ respectively. In contrast, for larger values ($n \geqslant 3$), the abscissa is given by $\frac{2n}{3}$.

A notable feature is that the abscissa of convergence grows without bound as $n$ increases. These groups, together with those appearing in the subsequent theorem, constitute the first known examples of a family of $p$-adic analytic groups exhibiting unbounded polynomial representation growth.

\begin{theorem}\label{thm:diagonal on symmetric square}
    Let~$\Op$, $p$, and $q$ be as above.
	For all $m \in \N$ that are soundly permissible for~$\mathfrak{sl}_2(\Op)$ and all $n \in \N_+$, the representation zeta function of the semi-direct product of the principal congruence subgroup $\SL_2^m(\Op)$ and $\Op^{3n}$ with respect to the $n$\textsuperscript{th} direct power of the symmetric square of the natural action is
	\[
	    \zeta_{\SL_2^m(\Op) \ltimes (\Sym^2(\Op^{2}))^n}(s) = \zeta_{\SL_2^m(\Op)}(s)\cdot q^{3nm}\frac{(1-q^{-s})(q+q^{-s})f(q,q^{-s})}{(1-q^{n+2-2s})(1-q^{3(n-s)})} 
	\]
	where
	\[
		f(q,t) = (1+q^nt^2)(1+qt) + t(1+q^{n+1}t).
	\]
	If $p = 2$, $\Op = \Z_2$ and $n = 1$, we have
	\[
		\zeta_{\SL_2^m(\Z_2) \ltimes \Sym^2(\Z_2^{2})}(s) = 2^{3m+1}\frac{(1-2^{-s})(2^{3-s}+(1-2^{-s}))}{1-2^{3-2s}}\cdot \zeta_{\SL_2^m(\Z_2)}(s).
	\]
\end{theorem}
In particular,
\[
	\zeta_{\SL_2^m(\Op) \ltimes \Sym^2(\Op^{2})}(s) = \zeta_{\SL_2^m(\Op)}(s)\cdot q^{3m} \frac{(1-q^{-s})^2(q+q^{-s})}{(1-q^{3-2s})(1-q^{1-s})}.
\]
Again, the abscissae of the groups $\SL_2^m(\Op) \ltimes (\Sym^2(\Op^{2}))^n$ with varying $n$ is potentially governed by both factors in the denominator, as well as unbounded: it is~$3/2$ for~${n = 1}$ and~$n$ otherwise.

To achieve \cref{thm:main}, we need a good understanding of the representation zeta functions of subgroups of $\SL_2^1(\Op)$. We say that two groups are called \emph{thetyspectral} if their representation zeta functions differ by a constant factor. 
To obtain \cref{thm:main}, we prove---see \cref{cor:subgroups}---that all open potent subgroups of~$\SL_2^m(\Op)$ for permissible~$m$ are thetyspectral. In fact, we investigate a more general situation, see \cref{thm:subgroup formula}, using a `weak' analogue of Lie lattice isomorphisms. It is unclear how to characterise thetyspectral subgroups in general. We present a class of subgroups of~$\SL_3(\Op)$ of the desired kind in \cref{prop:thethys in sl3}; the method naturally generalises to cover other semi-simple compact $p$-adic analytic groups.

On the other hand, we prove that not all open potent subgroups of~$\SL_2(\Op)\ltimes \Op^2$ are thetyspectral by the explicit computation of the representation zeta function of a certain subgroup.
\begin{theorem}
    Let $p$ be an odd prime and let~$\Op$ be an unramified extension of~$\mathbb{Z}_p$. For~$k\in \mathbb{N}$, consider the group~$H_k = S_k \ltimes \Op^2$, where
    \[
        S_k= \left\{M\in \SL_2^1(\Op)\mid 
            M \equiv \begin{pmatrix}
		    1 & \ast \\
		    0 & 1
		\end{pmatrix} \bmod p^{k+1}\right\}.
    \]
    The groups $\SL_2^1(\Op) \ltimes \Op^2$ and $H_k$ are not thetyspectral. In particular,
    \[
        \zeta_{H_k}(s) = p^{2k+2}\frac{(1-p^{-s})((1-(p^{1-s})^{k+1}) + p^{1-2s}(1-p^{(1-s)(k-1)})}{(1-p^{1-s})^2(1+p^{1-s})}
		\cdot \zeta_{\SL_2^1(\Op)}(s).
    \]
\end{theorem}

\subsection*{Acknowledgements} Some findings of this paper were presented in the respective doctoral theses of the authors, both supervised by Benjamin Klopsch. The authors express their gratitude to him for suggesting the research topic, providing guidance, and offering feedback on earlier drafts.

\section{Preliminaries}

\subsection{General notation}
\label{sub:general_notation}
We write $\N$ for the set of positive integers and we put $\N_0=\N\cup \{0\}$. The set $\{1, \dots, n\}$ is denoted $[n]$.

\subsection{Representation theory} 
\label{sub:representation_theory}

All groups considered in this section are profinite. For a group $G$, denote by~$\Irr(G)$ the set of isomorphism classes of finite-dimensional irreducible continuous complex representations, and, for any $n \in \N$, by $\Irr_n(G)$ the subset of representations of dimension~$n$. There is a bijection between the isomorphism classes of irreducible representations of~$G$ and its irreducible complex characters. The trivial representation is denoted $\idrep_G$. A representation~$\sigma \colon G \to \GL(V)$ is called \emph{smooth} if the map~$G\times V\to V$ is continuous, where~$W$ is equipped with the discrete topology. Every smooth representation decomposes as a direct sum of smooth irreducible representations, and the smooth irreducible representations of~$G$ are precisely the finite-dimensional irreducible continuous representations of~$G$, see for example~\cite[Lemma 2.1]{KK19} and the references therein. In the following, all representations are finite-dimensional and smooth.

Given two representations~$\sigma \colon G \to \GL(V)$ and~$\varphi\colon G \to \GL(W)$ of~$G$, the \emph{tensor product}~$\sigma \otimes \varphi$ is the representation of~$G$ on~$V\otimes_\mathbb{C} W$ satisfying ${((\sigma\otimes \varphi)(g))(v \otimes w)} = \sigma(g).v \otimes \varphi(g).w$ for all $g\in G, v \in V$, and $w \in W$. The \emph{symmetric square} $\Sym^2(\sigma) \colon G \to \GL(\Sym^2(V))$ of a representation $\sigma \colon G \to \GL(V)$ is the restriction of the tensor product $\sigma \otimes \sigma$ to the subspace of symmetric tensors.

Let~$\mathcal{C}^\infty(G,V)$ denote the space of continuous functions from~$G$ to a~$\mathbb{C}$-vector space~$V$, equipped with the discrete topology. Given a closed subgroup~$H \leq G$ and a representation~$\sigma$ of~$H$, define $V_{\sigma} = \{ f\in \mathcal{C}^\infty(G,V) \mid \forall\, h\in H,\,\forall \, x \in G : f(hx) = \sigma(h).f(x)\}$. The \emph{induced representation}~$\Ind_H^G(\sigma) \colon G \to \GL(V_\sigma)$ of $\sigma$~on $G$ is defined by right translation $(\Ind_H^G(\sigma)(g))(f)(x) = f(xg)$. Following Kionke and Klopsch \cite[Section 2.3]{KK19}, we refer to the representation as `induced’ rather than `co-induced’, which is the terminology used by Serre in \cite[Section I 2.5]{Ser02}.

Given a group~$H$, a homomorphism~$f \colon G \to H$, and a representation~$\sigma$ of~$H$, the \emph{inflation of~$\sigma$ along~$f$} is the representation~$\Inf_H^{G, f}(\sigma) = f\sigma$ of $G$. If the choice of morphism is clear from the context, we shall omit related the superscript. We use the same notation for the characters of induced or inflated representations.

Let~$H \leq G$ be a closed subgroup, and let~$\sigma$ be a representation of~$H$. We say that~$\sigma$ is \emph{extendable} if there exists a representation~$\tilde{\sigma}$ of~$G$ such that~$\tilde{\sigma}|_H$ equals~$\sigma$.

For the remainder of this section, let~$H \leq G$ be a closed subgroup such that~$G$ decomposes (continuously) as a semi-direct product~$G =  H \ltimes V$, where~$V$ is closed and abelian. We describe the irreducible representations of~$G$ in terms of those of~$V$, using the classic description by Mackey \cite{Mac58} for general semi-direct products. 

In this set-up, the subgroup $H$ acts on $\Irr(V)$ by conjugation on the argument. Denote by $\Stab_H(\sigma)$ the point stabiliser of $\sigma \in \Irr(V)$ under this action. Then $\sigma$ is extendable to the subgroup~$H_\sigma = \Stab_H(\sigma) \ltimes V$ of $G$, as witnessed by the representation $\Ext_V^{H_\sigma}(\sigma) (h v ) = \sigma(v)$. Using this terminology, the structural form of irreducible representations of~$G$ may be described, cf.\ \cite[Proposition 25]{Ser96}.

\begin{proposition}\label{prop:reps of sd prod}
	Let~$G$ be the semi-direct product~$H\ltimes V$. Let~$\mathcal{X}$ be a set of representatives of the orbits of the action of~$H$ on~$\Irr(V)$. For each~$\chi \in \mathcal{X}$, let~$S_\chi$ be its point stabiliser in~$H$. 
	If $|H : S_\chi| < \infty$ for all $\chi \in \mathcal{X}$, and if the set~$\mathcal{X}$ is countable, then every irreducible representation of~$G$ is of the form
    \begin{equation*}
        \Ind_{ S_\chi\ltimes V }^{G}\left(\Inf_{S_\chi}^{ S_\chi\ltimes V }(\tau) \otimes \Ext_{V}^{ S_\chi \ltimes V}(\chi)\right),
    \end{equation*}
    for some~$\chi \in\mathcal{X}$, and some~$\tau \in \Irr(S_\chi)$. Two representations of this form are equivalent if and only if they are given by the same pair $\chi, \tau$.
\end{proposition}

For profinite (or more generally, compact) groups, the usual inner product on the set of characters associated to irreducible representations for finite groups is generalised by $\langle \chi, \theta \rangle_G = \int_G \chi(g) \overline{\theta(g)} \mathrm{d}\mu(g)$, where $\chi,\theta$ denote irreducible characters, and~$\mu$ denotes the normalised (left-)Haar measure of~$G$. As for finite groups, given an irreducible component~$\theta$ of~$\chi$, the value of~$\langle \chi, \theta \rangle_G$ equals the multiplicity of~$\theta$ appearing in the decomposition of~$\chi$. For us, the following equality will be of use.

\begin{proposition}\label{prop:consts of idrep}
	Using the terminology of the previous preposition, let~$\chi \in \mathcal{X}$ and~$\tau$ be an irreducible representation of~$S_\chi$. Denote by~$\theta_\tau$ and $\theta_{\tau,\chi}$ the characters associated to~$\tau$ and to the representation~$\Inf_{S_\chi}^{ S_\chi\ltimes V }(\tau) \otimes \Ext_{V}^{ S_\chi \ltimes V}(\chi)$ of~$S_\chi\ltimes V~$, respectively. Then 
    \[
        \langle \Ind_{ S_\chi \ltimes V}^{G}(\theta_{\tau,\chi}), \Ind_H^G(\idrep_H) \rangle_G = \langle \Ind_{ S_\chi }^{H} (\theta_\tau), \idrep_{H} \rangle_{H}.
    \]
\end{proposition}

\begin{proof}
	Let~$\mathcal{R}_{S_\chi}$ be a set of coset representatives of~$S_\chi$ in~$H$, which also serves as a set of coset representatives of~$S_\chi \ltimes V$ in~$G$. For~$x \in \mathcal{R}_{S_\chi}$ and~$h \in H$, the conjugate~$h^x$ is contained in~$S_\chi \ltimes V$ if and only if it is contained in~$S_\chi$. Since~$H$ is closed, by \cite[Theorem 6.10]{Fol95}, Frobenius~reciprocity yields
	\belowdisplayskip=-11pt
    \begin{align*}
        \langle \Ind_{ S_\chi \ltimes V }^{G}(\theta_{\tau,\chi}), \Ind_H^G(\idrep_H) \rangle_G
         &=\langle \Res_H^{G}\Ind_{ S_\chi \ltimes V}^G(\theta_{\tau,\chi}), \idrep_H \rangle_H\\
        &= \int_H \sum_{\substack{x \in\mathcal{R}_{S_\chi} \\(h)^x \in  S_\chi \ltimes V }} \theta_{\tau,\chi}(h^x) \mathrm d \mu(h)\\
        &=\int_H \sum_{\substack{x \in\mathcal{R}_{S_\chi} \\h^x \in S_\chi}} \theta_\tau(h^x) \cdot \chi(1) \mathrm d \mu(h)\\
         &= \int_H \Ind_{S_\chi}^H(\theta_\tau)(h) \cdot\idrep_H(h)\, \mathrm d \mu(h)
         = \langle \Ind_{S_\chi}^H(\theta_\tau), \idrep_H \rangle_H.\qedhere
    \end{align*}
\end{proof}

\subsection{Representation zeta functions} 
\label{sub:representation_zeta_functions}

Let $G$ be an abstract group such that $|\Irr_n(G)|$ is finite, for all $n \in \N$. The \emph{representation zeta function of $G$} is the Dirichlet generating function
\[
	\zeta_G(s) = \sum_{\varphi \in \Irr(G)} \dim(\varphi)^{-s} = \sum_{n \in \N} |\Irr_n(G)| n^{-s}.
\]
Changing the prospective, Kionke and Klopsch \cite{KK19} introduced a more general zeta function that can be associated to any ‘suitably tame’ infinite-dimensional representation of a group.

A smooth representation~$\sigma$ of a profinite group~$G$ is called \emph{strongly~admissible}, if its decomposition into irreducible constituents $\sigma = \bigoplus_{\varphi \in \Irr(G)} m(\sigma, \varphi) \varphi$ satisfies
\[
	\sum_{\varphi \in \Irr_n(G)} m(\sigma, \varphi) < \infty
\]
for all~$n \in \N$. To a strongly admissible representation~$\sigma$ of~$G$, one associates the formal Dirichlet generating function
\[
    \zeta_\sigma(s) = \sum_{\varphi \in \Irr(G)} m(\sigma, \varphi) \dim(\varphi)^{-s},
\]
which is called the \emph{zeta function associated to~$\sigma$}.

For the regular representation $\rho=\Ind_{1}^{G}(\idrep)$, which contains every~$\varphi \in \Irr(G)$ with multiplicity~$\dim(\varphi)$, one finds
\begin{equation*}
    \zeta_\rho(s) = \sum_{\varphi \in \Irr(G)} \dim(\varphi)^{1-s} = \zeta_G(s-1),
\end{equation*}
provided $\rho$ is strongly admissible; in this way, the zeta functions of strongly admissible representations generalise the representation zeta functions of groups, see \cite[Example~2.5]{KK19}.

Let~$G$ be a finitely generated profinite group and let~$H \leq G$ be a closed subgroup. We are concerned with representations of the form~$\Ind_H^G(\idrep_H)$; to simplify our notation we set
\[
	\zeta_{H}^G (s) = \zeta_{\Ind_H^G(\idrep_H)}(s)
\]
and call this zeta function the \emph{representation zeta function of $G$ relative to $H$}. By~\cite[Theorem A]{KK19}, the representation~$\Ind_H^G(\idrep_H)$ is strongly admissible if and only if the group~$G$ is \emph{FAb relative to~$H$}, i.e.\ if the quotient~$K/(H\cap K)[K,K]$ is finite for every open subgroup~$K$ of~$G$.


\subsection{Saturable and uniformly potent pro-\emph{p} groups} 
\label{sub:uniformly_potent_pro_p_groups}

If~$p$ is an odd prime, a pro-$p$ group~$G$ is called \emph{potent} if~the $(p-1)$\textsuperscript{st} term of the lower central series $\gamma_{p-1}(G)$ is contained in the the subgroup~$G^{p}$ generated by~$p$\textsuperscript{th} powers. For~$p = 2$, the group $G$ is considered potent if~${[G,G] \leq G^4}$. The group $G$ is called \emph{powerful} if $[G,G] \leq G^p$ for odd primes, or if $[G,G] \leq G^4$ for $p = 2$. It is evident that every powerful group is potent. A finitely generated, potent (resp.\ powerful)  and torsion-free pro-$p$ group is called \emph{uniformly~potent} (resp.\ uniformly powerful); such groups are \emph{saturable} in the sense of Lazard, see~\cite{Gon07,Klo05,Laz65}.

Let $R$ be a commutative ring. An $R$-Lie lattice is a free $R$-module of finite rank endowed with a Lie bracket. To a uniformly potent pro-$p$ group~$G$, one can associates a potent~$\mathbb{Z}_p$-Lie lattice~$\mathfrak{g} = \log(G)$, which coincides with~$G$ as a topological space. See~\cite{GK09} for a treatment of the Lie correspondence between~$G$ and~$\mathfrak{g}$. Here, a~$\Z_p$-Lie lattice~$\mathfrak{g}$ is called potent if~$\gamma_{p-1}(\mathfrak{g})\subseteq p\, \mathfrak{g}$ for odd primes~$p$ and~$\gamma_{2}(\mathfrak{g})\subseteq 4 \,\mathfrak{g}$ in case~$p=2$.

Let~$\Op$ be a compact discrete valuation ring of characteristic~$0$ and residue characteristic~$p$, with valuation ideal~$\p$, and let~$K$ be the field of fractions of~$\Op$, which constitutes a finite extension of~$\Q_p$. Given a $\mathbb{Z}_p$-Lie lattice $\mathfrak{g}$, write~$\mathfrak{g}_\mathfrak{p}$ for the tensor product~$\mathfrak{g}\otimes_{\Z_p} \Op$. If $\mathfrak{g}$ is potent, so is~$\mathfrak{g}_\mathfrak{p}$.
For~${m\in \mathbb{N}_0}$, its principal congruence sublattices are~$\mathfrak{g}_{\mathfrak{p},m}=\pi^m \cdot\mathfrak{g}_\mathfrak{p}$. 
The~$\Op$-Lie lattice~$\mathfrak{g}_{\mathfrak{p},m}$ is called \emph{potent} if it is potent as a~$\Z_p$-Lie lattice.
The Hausdorff series defines a group multiplication on~$\mathfrak{g}_{\mathfrak{p},m}$ allowing us to define the group~$\exp(\mathfrak{g}_{\mathfrak{p},m})$.
The Lie lattice~$\mathfrak{g}_{\mathfrak{p},m}$ is potent for all sufficiently large integers~$m$, so that~$G_{\mathfrak{p},m} = \exp(\mathfrak{g}_{\mathfrak{p},m})$ is a uniformly potent pro-$p$ group. We call such positive integers~$m$ \emph{permissible} for~$\mathfrak{g}$, cf.\ \cite[Definition~2.2]{AKOV13}. For a given $\Op$-lattice~$\mathfrak{g}$, almost all positive integers are permissible. Indeed, denote by $e=e(\Op,\Z_p)$ the absolute ramification index of $\Op$. Then, for odd primes, every $m\geqslant e/ (p-2)$ is permissible, and for $p = 2$, every $m\geqslant 2e$ is permissible, cf.\ \cite[Proposition~2.3]{AKOV13}. In particular, if~$\Op$ is unramified over~$\mathbb{Z}_p$ and~$p$ is odd, every~$m\in \mathbb{N}$ is permissible for every~$\Op$-Lie lattice~$\mathfrak{g}$, cf.\ \cite[Section~2.1]{AKOV13}.

In order to deal with semi-direct products, i.e., to ensure that the $\Op$-Lie lattice~$\mathfrak{g}_{\mathfrak{p},m}$ is potent, we need to impose the following stronger condition.

\begin{definition}   
    Let $\Op$ be a compact discrete valuation ring of characteristic~$0$, and valuation ideal~$\mathfrak{p}$, and let $\mathfrak{h}$ be an $\Op$-Lie lattice. We call $m \in \N_0$ \emph{soundly permissible} for~$\mathfrak{h}$ if for every faithful $\Op$-Lie lattice homomorphism $\sigma \colon \mathfrak{p}^m\mathfrak{h} \to \mathfrak{p}^m\mathfrak{gl}_n(\Op)$, with~${n \in \N_0}$, the corresponding semi-direct sum $\mathfrak{p}^m(\mathfrak{h} \oplus \Op^n)$ is potent.
\end{definition}

It is immediate that if $m$ is soundly permissible for $\mathfrak{h}$, it is in particular permissible. We now show that for every $\Op$-Lie lattice $\mathfrak{h}$, every sufficiently large $m \in \N$ is soundly permissible.

\begin{lemma}\label{lem:semidirect potent}
    Let $\Op$ be a compact discrete valuation ring of characteristic $0$ and residue characteristic~$p$, let $e = e(\Op, \mathbb{Z}_p)$ be the absolute ramification index, and let $\mathfrak{h}$ be an $\Op$-Lie lattice. If $p$ is odd, every $m \geq e$ is soundly permissible. If $p$ is $2$, every $m \geq 2e$ is soundly permissible. In particular, for every faithful $\Op$-Lie lattice homomorphism $\sigma \colon \mathfrak{p}^m\mathfrak{h} \to \mathfrak{p}^m\mathfrak{gl}_n(\Op)$, the group $\exp(\mathfrak{p}^m(\mathfrak{h} \oplus \Op^n))$ is uniformly potent.
    
    Furthermore, in the case that $\exp(\mathfrak{p}^m\mathfrak{h})$ is a uniformly powerful group for some integer~$m$ satisfying the condition above, the group $\exp(\mathfrak{p}^m(\mathfrak{h} \oplus \Op^n))$ is also uniformly powerful.
\end{lemma}

\begin{proof}
    Note that under the given assumption on $m$, it is permissible for $\mathfrak{h}$ by \cite{AKOV13}. Here we use that, for $p = 2$, the integer is at least $2e$.
    
    Let $\sigma \colon \mathfrak{p}^m \mathfrak{h} \to \mathfrak{p}^m \mathfrak{gl}_n(\Op)$ be a faithful Lie lattice homomorphism and denote the corresponding semi-direct sum $\mathfrak{h} \oplus \Op^n$ by $\mathfrak{g}$. Assume that $p$ is odd. The~$(p-1)$\textsuperscript{st} term of the lower central series of $\mathfrak{p}^m\mathfrak{g}$ is generated by~$\gamma_{p-1}(\mathfrak{p}^m\mathfrak{h})$ and~$[\mathfrak{p}^m\Op^n, \mathfrak{p}^m\mathfrak{h}, \dots, \mathfrak{p}^m\mathfrak{h}]$, since~$[\mathfrak{p}^m\Op^n, \mathfrak{p}^m\mathfrak{h}] \subseteq \mathfrak{p}^m\Op^n$, which is abelian. Since $m$ is permissible for $\mathfrak{h}$, we find $\gamma_{p-1}(\mathfrak{p}^m\mathfrak{h}) \subseteq \mathfrak{p}^{m+1}\mathfrak{h}$. Let $h \in \mathfrak{p}^m\mathfrak{h}$ and $v \in \mathfrak{p}^m\Op^n$. Then
    \[
		[v, h] = \sigma(h)(v) \in v \cdot \mathfrak{p}^m \mathfrak{gl}_n(\Op),
	\]
    whence $\gamma_{p-1}(\mathfrak{p}^m\mathfrak{g})$ is a subset of $\mathfrak{p}^{2m} \Op^n$. Since $m \geq e$, also $\mathfrak{p}^{2m} \Op \subseteq (\mathfrak{p}^m\Op)^p$ and $\mathfrak{p}^m\mathfrak{g}$ is potent. The case $p = 2$ and the statements regarding uniformly powerful groups follows in the same way.
\end{proof}

In the realm of finitely generated profinite groups, representation rigid groups are algebraically characterised by the FAb property: a topological group is FAb if every open subgroup~$H$ has finite abelianisation~$H/[H,H]$, see~\cite[Proposition~2]{BLMM02}.
For a~$p$-adic analytic group~$G$, this characterisation can be carried over to its $p$-adic Lie algebra, employing Lazard's correspondence between saturable pro-$p$ groups and saturable~$\mathbb{Z}_p$\nobreakdash -Lie lattices, cf.\ \cite{Laz65,GK09}. In particular, a~$p$-adic analytic group~$G$ is FAb, and hence representation rigid, if and only if it has an open FAb saturable pro-$p$ subgroup~$U$; and a saturable pro-$p$ subgroup~$U$ is FAb if and only if the~$\mathbb{Z}_p$\nobreakdash-Lie lattice~$\log(U)$ is FAb. Furthermore, the latter happens if and only if the ~$\Q_p$-Lie algebra~$\log(U)\otimes_{\mathbb{Z}_p}\mathbb{Q}_p$ is perfect,~cf.\ \cite[Proposition 2.1]{AKOV13}.


\subsection{The \emph{p}-adic integration formalism} 
\label{sub:the_p_adic_integration_formalism}

Making use of the Kirillov orbit method~\cite{Gon09}, Klopsch and Kionke \cite[Section~4]{KK19} describe the relative representation zeta functions of uniformly potent pro-$p$ groups and their potent subgroups using a $p$-adic integral as follows.

Let $\mathfrak{g}$ be an $\Op$-Lie lattice with $\Op$-basis $\mathfrak{Z}$ and let $\mathfrak{X}, \mathfrak{Y}$ be (ordered) subsets of $\mathfrak{Z}$. Write $g \mapsto g^\ast$ for the map from $\mathfrak{g}$ to its dual $\mathfrak{g}^\ast$ induced by the mapping of $\mathfrak{Z}$ to its dual basis. The \emph{commutator matrix of~$\mathfrak g$ with respect to $\mathfrak{X}$ and $\mathfrak{Y}$} is the $|\mathfrak{X}|$-by-$|\mathfrak{Y}|$ matrix $\Com(\mathfrak{X}, \mathfrak{Y})$ of $\Op$-linear forms with entries $[x, y]^\ast$ for $x$ running through $\mathfrak{X}$ and~$y$ running through $\mathfrak{Y}$.  If $\mathfrak{X} = \mathfrak{Y}$, this matrix is skew-symmetric. Write $\Com(\mathfrak{X}, \mathfrak{Y})(g)$ for the matrix whose entries are the entries of the commutator matrix applied to an element $g \in \mathfrak{g}$.

The determinant, restricted to the set of $n$-by-$n$ skew-symmetric matrices, is the square (as a polynomial) of the \emph{Pfaffian determinant}, c.f.\ e.g.\ \cite[Chapter 9.5]{Igu00}. A \emph{Pfaffian minor} of a skew-symmetric matrix is the Pfaffian determinant of a principal sub-matrix. We write~$\Pfaff(T)$ for the set of all Pfaffian minors of a skew-symmetric matrix~$T$, and $\operatorname{Minor}(A)$ for the set of all (not necessarily principal) minors of a matrix~$A$.

Let~$|\cdot|_\mathfrak{p}$ denote the~$\mathfrak{p}$-adic norm on $K$. For a subset~$S \subseteq K$, define
\[
	\lVert S \rVert_{\p}= \max\{ |s|_{\p} \mid s \in S \}.
\]

\begin{theorem}\textup{\cite[Proposition C and Proposition~4.6]{KK19}} \label{thm:rel zeta integral}
    Let~$\Op$ be a compact discrete valuation ring of characteristic~$0$, residue characteristic~$p$, and of residue field cardinality~$q$. Let~$\mathfrak{p}$ be the valuation ideal, and let $\pi$ be a uniformiser, let~$K$ be the field of fractions of~$\Op$, let~$\mathfrak{g}$ be an~$\Op$-Lie lattice and  let~$m\in \mathbb{N}$ be such that~$G=\exp(\pi^m\mathfrak{g})$ is a uniformly potent pro-$p$ group, and let~$\mathfrak{h}$ be direct summand of the~$\Op$-Lie lattice corresponding to a potent subgroup~$H=\exp(\pi^m\mathfrak{h})\leqslant_c G$ such that~$G$ is FAb relative to~$H$. Write~$d = \dim_{\Op} \mathfrak g - \dim_{\Op} \mathfrak h$. Let~$\mathfrak{Y}$ be a basis of~$\mathfrak{h}$ and let~$\mathfrak{Z} = \mathfrak{X} \cup  \mathfrak{Y}$ be an extension to an~$\Op$-basis of~$\mathfrak{g}$. Then
        \[
            \zeta^G_{H}(s) =  q^{md}\int\limits_{K\mathfrak{X}} \lVert \Pfaff( \Com(\mathfrak{Z}, \mathfrak{Z})(w))\rVert_{\p}^{-1-s} \mathrm{d}\mu(w),
        \]
    where~$\mu$ denotes the Haar~measure of~$K\mathfrak{X}$ normalised with respect to~$\mathfrak{g}$.
\end{theorem}

Note that in comparison to \cite{KK19}, our formula is `dualised'. In \cite{KK19}, the commutator matrix has entries in~$\mathfrak{g}$ and the area of integration is the dual of $K\mathfrak{X}$.

Recall that the representation zeta function of a suitable group $G$ fulfils the equation $\zeta_G(s) = \zeta^G_1(s+1)$. In this way, the theorem above generalises similar formulae found in \cite{AKOV13, Jai05}.

Since we are interested in the special case where $G$ is a semi-direct product of some group $H$ with an abelian group, we can somewhat simplify the formula above.

\begin{proposition}\label{prop:integral formula-semidirect case}
    Using the same notation as in \cref{thm:rel zeta integral}, assume that $\mathfrak{g}$ is a semi-direct sum $\mathfrak{h} \oplus \Op^n$ with respect to a faithful $\Op$-Lie lattice homomorphism $\sigma \colon \mathfrak{p}^m \mathfrak{h} \to \mathfrak{p}^m \mathfrak{gl}_n(\Op)$ for $m$ soundly permissible for $\mathfrak{h}$, and assume that $\mathfrak{X}$ generates~$\Op^n$. Then the representation zeta function of $G = \exp(\mathfrak{p}^m\mathfrak{g}) \cong H \ltimes \Op^n$ relative to $H = \exp(\mathfrak{p}^m \mathfrak{h})$ satisfies
    \begin{align*}
        \zeta^G_{H}(s) &=q^{mn}\int_{K\mathfrak{X}} \lVert \operatorname{Minor}(\Com(\mathfrak{X}, \mathfrak{Y})(w))\rVert_{\p} ^{-1-s}\mathrm d \mu(w),
    \end{align*}
    where~$\mu$ denotes the Haar~measure of~$K\mathfrak{X}$ normalised with respect to~$\mathfrak{g}$.   
\end{proposition}

\begin{proof}
	In view of \cref{thm:rel zeta integral}, consider the commutator matrix $\Com(\mathfrak{Z}, \mathfrak{Z})$. Since $V$ is abelian, also~$\mathfrak{v}$ is abelian, the commutator matrix has block form
	\[
		\Com(\mathfrak{Z}, \mathfrak{Z}) = \begin{pmatrix}
        	0 & \Com(\mathfrak{X}, \mathfrak{Y})\\
    		\Com(\mathfrak{Y}, \mathfrak{X}) & \Com(\mathfrak{Y},  \mathfrak{Y})
    	\end{pmatrix}.
	\]
	Naturally, $\Com(\mathfrak{Y}, \mathfrak{X}) = -\Com(\mathfrak{X}, \mathfrak{Y})^\intercal$. Since $\mathfrak{h}$ is closed under the Lie bracket, the entries of $\Com(\mathfrak{Y},  \mathfrak{Y})$ are functional of $\mathfrak{h}$ and $K\mathfrak{X}$ is contained in their kernel, whence, for any $w \in K\mathfrak{X}$, the matrix $\Com(\mathfrak{X}, \mathfrak{Y})(w)$ has the form
	\[
		\begin{pmatrix}
			0 & \Com(\mathfrak{X}, \mathfrak{Y})(w)\\
			-\Com(\mathfrak{X}, \mathfrak{Y})^\intercal(w) & 0
		\end{pmatrix}.
	\]
	Its principal sub-matrices are again of the shape~$\left(\begin{smallmatrix}0 & A \\ -A^\intercal & 0\end{smallmatrix}\right)$, where $A$ is a (generally non-principal) sub-matrix of~$\Com(\mathfrak{X}, \mathfrak{Y})(w)$.

	The determinant of such a block matrix is equal $\det(A)^2$, whence $\Pfaff(\Com(\mathfrak{Z}, \mathfrak{Z})(w))$ is equal to $\operatorname{Minor}(\Com(\mathfrak{X}, \mathfrak{Y})(w))$. Reviewing the integral formula, we have proven the desired statement.
\end{proof}


\section{Zeta functions of subgroups and semi-direct products}

\subsection{Thetyspectral groups}

In the study of certain instance of a zeta function (i.e., a specific counting problem) related to groups, it is a common theme to consider when two non-isomorphic groups yield the same zeta function. Two such groups are referred to as \emph{isospectral}; for instance, see \cite{LS03}. In the case of the representation zeta function, a similar but generally distinct behaviour is exhibited by different congruence subgroups of the same uniformly potent pro-$p$ group: the resulting zeta functions are not equal but only differ by a constant factor; as readily observed from \cref{thm:rel zeta integral} and \cite[Proposition 6]{GJK14}. To facilitate the discussion of this phenomenon, we recall the previously introduced terminology.

\begin{definition}
    Let $G$ be a group, let $H$ be a subgroup of $G$ and let $\lambda \in \mathbb{N} $ be a constant. We say that $H$ is \emph{thetyspectral with factor $\lambda$} if $\zeta_H (s)= \lambda \zeta_G(s)$.
\end{definition}

We establish a condition that identifies certain thetyspectral subgroups by describing the representation zeta function as a $p$-adic integral and leveraging the Lie correspondence, which states that a group and its finite index subgroups correspond to lattices of the same dimension. Consequently, taking tensor products with the field of fractions one obtains isomorphic structures.

\begin{theorem}\label{thm:subgroup formula}
	Let~$\Op$ be a compact discrete valuation ring of characteristic~$0$, residue characteristic~$p$, and of residue field cardinality~$q$. Let~$\mathfrak{p}$ be the valuation ideal, and let~$\pi$ be a uniformiser, let~$K$ be the field of fractions of~$\Op$, let~$\mathfrak{g}$ be an~$\Op$-Lie lattice, and let~$m\in \mathbb{N}$ be such that~$G=\exp(\pi^m\mathfrak{g})$ is a uniformly potent pro-$p$ group. Let~$\mathfrak{h}$ be an open potent sub-$\Op$-Lie lattice of $\mathfrak{g}$, corresponding to a potent subgroup~$H=\exp(\pi^m\mathfrak{h})\leqslant_c G$.
	
	Write $\mathfrak{g}_K = \mathfrak{g} \otimes_{\Op} K$ and~$\mathfrak{h}_K = \mathfrak{h} \otimes_{\Op} K$, and let~$\beta \colon \mathfrak{g}_K\wedge\mathfrak{g}_K \to \mathfrak{g}_K$ be the linear map induced by the Lie bracket of~$\mathfrak{g}$. Fix a $K$-linear isomorphism~$\xi \colon \mathfrak{g}_K \to \mathfrak{h}_K$ that restricts to an~$\Op$-linear map $\mathfrak{g} \to \mathfrak{h}$.

    If there exists a~$K$-linear map~$\psi \colon \mathfrak{g}_K \to \mathfrak{h}_K$
    such that
    \[  
        \beta(\xi\wedge\xi) = \psi\beta,
    \]  
    then $H$ is thetyspectral with factor~$|\det(\psi\xi^{-1})|_{\p}^{-1}$.
\end{theorem}

\begin{proof}
    We use the integral formalism introduced earlier, whence
	\[
		\zeta_G(s) = q^{m \dim_{\Op} \mathfrak{g}} \int_{K \mathfrak{X}} \lVert \Pfaff(\Com(\mathfrak{X}, \mathfrak{X})(w)) \rVert_{\p}^{-2-s} \mathrm{d}\mu(w),
	\]
	where~$\mathfrak{X}$ is an~$\Op$-basis for~$\mathfrak{g}$. Naturally, $\xi(\mathfrak{X})$ is an $\Op$-basis for $\mathfrak{h}$, and the entries of~$\Com(\xi(\mathfrak{X}), \xi(\mathfrak{X}))$ are of the form
	\[
		\mathrm{d}^{\xi(\mathfrak{X})}([\xi(x), \xi(x')]) = \mathrm{d}^{\xi(\mathfrak{X})}(\beta(\xi \wedge \xi)(x \wedge x')) = \mathrm{d}^{\xi(\mathfrak{X})}(\psi\beta(x \wedge x')) = \mathrm{d}^{\xi(\mathfrak{X})}(\psi[x, x'])
	\]
	for $x, x' \in \mathfrak{X}$, whence for $w \in K\mathfrak{X}$ we find
	\[
		\mathrm{d}^{\xi(\mathfrak{X})}([\xi(x), \xi(x')])(w) = \mathrm{d}^{\mathfrak{X}}([x, x'])(\psi^\ast\xi^{-1}(w)).
	\]
	Thus we find
	\begin{align*}
		\zeta_H(s) &= q^{m \dim_{\Op} \mathfrak{h}} \int_{ \xi(K\mathfrak{X})} \lVert \Pfaff(\Com(\xi(\mathfrak{X}), \xi(\mathfrak{X}))(w)) \rVert_{\p}^{-2-s} \mathrm{d}\mu(w)\\
		&= q^{m \dim_{\Op} \mathfrak{g}} \int_{K\mathfrak{X}} \lVert \Pfaff(\Com(\mathfrak{X}, \mathfrak{X})(\psi^\ast\xi^{-1}(w))) \rVert_{\p}^{-2-s} \mathrm{d}\mu(w),
	\end{align*}
	using $\dim_{\Op} \mathfrak{h} = \dim_{\Op} \mathfrak{g}$ and $\xi(K\mathfrak{X}) = K \mathfrak{X}$. Now a simple change of variables ${w \mapsto \psi^\ast\xi^{-1}(w)}$ yields
	\[
		\zeta_H(s) = |\det(\psi^\ast \xi^{-1})|_{\p}^{-1} \zeta_G(s).\qedhere
	\]
\end{proof}

\begin{corollary}\label{cor:subgroups}
	Let $\Op$ be a compact discrete valuation ring of characteristic $0$ and residue characteristic $p$, with uniformiser $\pi$. Let $\mathfrak{g}$ be a $3$-dimensional simple $\Op$-Lie lattice and let $\mathfrak{h}$ be an open $\Op$-Lie sublattice of $\mathfrak{g}$. Let $m$ be such that $G = \exp(\pi^m \mathfrak{g})$ is a uniformly potent pro-$p$ group with open potent subgroup $H = \exp(\pi^m \mathfrak{h})$. Then
    \[
        \zeta_H(s) = |G : H| \cdot \zeta_{G}(s).
    \]
\end{corollary}

\begin{proof}	
	Write~$\mathfrak{g}_K=\mathfrak{g} \otimes_{\Op} K$, where~$K$ denotes the field of fractions of $\Op$. Since the linear map~$\beta \colon \mathfrak{g}_K \wedge \mathfrak{g}_K \to \mathfrak{g}_K$ induced by the Lie bracket of $\mathfrak{g}$ is non-degenerate, the image of~$\mathfrak{g} \wedge \mathfrak{g}$ is of finite index in~$\mathfrak{g}$.
	Since~$\dim_{\Op}(\mathfrak{g}) = 3$, the dimension of the exterior square~$\mathfrak{g} \wedge \mathfrak{g}$ is also~$3$. Thus as a~$K$-linear map,~$\beta$ is invertible.
	Denote by $\xi \colon \mathfrak{g}_K \to \mathfrak{h}_K$ a $K$-linear isomorphism restricting to a $\Op$-linear bijection $\mathfrak{g} \to \mathfrak{h}$.	Clearly the~$K$-linear map~$\psi \colon \mathfrak{g}_K \to \mathfrak{h}_K$ defined by the conjugate
	\[
		\psi =\beta(\xi \wedge \xi) \beta^{-1}
	\]
	meets the conditions of~\cref{thm:subgroup formula}. Its determinant satisfies
	\[
		\det(\psi) = \det(\xi \wedge \xi) = \det(\xi)^{\dim_{\Op} (\mathfrak{g}) - 1} = \det(\xi)^2,
	\]
	thus~$\det(\psi\xi^{-1}) = \det(\xi)$ and we have~$|\det(\psi\xi^{-1})|_\mathfrak{p}^{-1} = |\pi^m\mathfrak{g}:\pi^m\mathfrak{h}| = {|G : H|}$.
\end{proof}

Note that there are two simple $3$-dimensional $\Z_p$-Lie algebras; the familiar $\mathfrak{sl}_2(\Z_p)$ and the algebra of vanishing reduced trace elements $\mathfrak{sl}_1(\Delta_p)$ in a $p$-adic quaternion algebra $\mathbb{D}_p$ with maximal $\Z_p$-order $\Delta_p$, see \cite{Klo03,Rie70,KLG97}. The representation zeta functions of both families of the corresponding $p$-adic analytic groups have been computed in \cite{AKOV12}.

\cref{thm:subgroup formula} gives another interpretation of the fact that in \cref{thm:rel zeta integral}, the congruence level $m$ enters only as the power of a constant factor. Note that this fact, first described in~\cite{Jai05}, has been used to much effect, e.g., the main result in~\cite{GJK14} uses it heavily. In the language of the previous theorem, the lattice corresponding to the $(m+k)$\textsuperscript{th} principal congruence subgroup is the $(\pi^k)$\textsuperscript{th} power of the lattice corresponding to the $m$\textsuperscript{th} principal congruence subgroup $\mathfrak{g}$. The map $\xi$ is thus given by scalar multiplication with~$\pi^k$ and
\[
	(v \wedge w)(\xi\wedge\xi)\beta = [\pi^k v, \pi^k w] = \pi^{2k} [v,w] =(v,w)\beta \xi^2,\quad \text{for } v, w \in \mathfrak{g}.
\]
Thus we may choose~$\xi^2$ as our~$\psi$. Clearly~$\det(\psi\xi^{-1}) = \pi^{dk}$ and~$|\det(\xi)|_{\mathfrak{p}}^{-1}= q^{dk}$.

While \cref{cor:subgroups} will be of most use for us, \cref{thm:subgroup formula} can be also used to compute the representation zeta functions of many subgroups of groups of dimension greater than~$3$. In the following we conduct some computations that allow us to describe some thetyspectral subgroups of~$G = \SL_3^m(\Op)$, with permissible~$m$ for~$\mathfrak{sl}_3(\Op)$. The method, however, is not tied to the group and can be used in greater generality. The group $G$ corresponds to the the~$\mathcal{O}$-Lie lattice $\mathfrak{g} = \pi^m \mathfrak{sl}_3(\Op)$ of~$3\times3$-matrices with trace zero, which has the $\Op$-basis $\mathfrak{X} = \{h_{12}, h_{23}, e_{12}, e_{13}, e_{23}, f_{21}, f_{31}, f_{32}\}$, where, for the pairs $i,j$ as described,
 \begin{align*}
    h_{ij} &= E_{i,i}-E_{j,j}, & e_{ij}&= E_{i,j}, & f_{ji}&= E_{j,i},
\end{align*} 
and where $E_{i,j}$ denotes the~$3\times3$-matrix with entry~$\pi^m$ at position~$(i,j)$ and zero otherwise.

We consider sublattices of $\mathfrak{g}$ generated by
\[
	\{\pi^{k_1} h_{12}, \pi^{k_2} h_{23}, \pi^{k_3} e_{12}, \pi^{k_4} e_{13}, \pi^{k_5} e_{23}, \pi^{k_6} f_{21}, \pi^{k_7} f_{31}, \pi^{k_8} f_{32}\},
\]
i.e.\ the case that the map $\xi$ is given by a diagonal matrix of the form $\diag(\pi^{k_1}, \dots, \pi^{k_8})$, for some~${k_1, \dots, k_8 \in \N_0}$. It is easy to see that~$\xi \wedge \xi$ may be represented by
\[\diag(\pi^{k_1+k_2}, \pi^{k_1+k_3}, \dots, \pi^{k_7+k_8}),\]
using the standard ordering for the basis induced by~$\mathfrak{X}$ on the exterior square. Comparing the matrix representation of the bracket~$\beta$ and the disturbed~$(\xi \wedge \xi)\beta$, we find that
\[
    [\xi(e_{12}),\xi(f_{21})] = \pi^{k_3+k_6} h_{12} = \psi(h_{12}) = \psi([e_{12}, f_{21}])
\]
has to hold for any $\psi$ as in \cref{thm:subgroup formula}, as $[e_{12}, f_{21}] = h_{12}$. At the same time, $[e_{13}, f_{31}] = h_{12}$, whence $\psi(h_{12}) = \pi^{k_4+k_7} h_{12}$ and $k_3+k_6 = k_3 + k_7$. Going through each occurrence of each variable in the commutator matrix, we obtain a linear set of equations for the integers $k_i$ for~$i \in [8]$, which may minimally be described by
\begin{equation}\label{eq:condition on k}
\begin{aligned}
    k_1 = k_2 &= k_6 - k_7 + k_8, & k_3 &= 2k_1 - k_6, \\ k_4 &= 2k_1 - k_7, & k_5 &= 2k_1 - k_8.
\end{aligned}
\end{equation}
Now, additionally, we need that $\xi(\mathfrak{g})$ is actually a sublattice, i.e.\ closed under the bracket, which translates into a list of inequalities between the integers $k_i$ with~$i \in [8]$. Together with \eqref{eq:condition on k}, they reduce to the single inequality that $k_1 \geq 0$, which holds vacuously by the choice of $k_1 \in \N_0$. 

\begin{proposition}\label{prop:thethys in sl3}
    Let~$p$ be an odd prime and let~$\Op$ be a compact discrete valuation ring of characteristic~$0$ with residue field cardinality~$q$, a power of~$p$, valuation ideal~$ \mathfrak{p}$, and choose an uniformiser~$\pi$. Write~$G=\SL_3^m(\Op)$, for permissible~$m\in \mathbb{N}$, and let~$H$ be the subgroup of~$G$ corresponding to the~$\Op$-Lie sublattice spanned by
    \[
        \pi^{k_1} h_{12}, \pi^{k_2} h_{23}, \pi^{k_3} e_{12}, \pi^{k_4} e_{13}, \pi^{k_5} e_{23}, \pi^{k_6} f_{21}, \pi^{k_7} f_{31}, \pi^{k_8} f_{32},
    \]
    where the elements $k_i \in \N_0$ for~$i \in \{1, \dots, 8\}$ denote non-negative integers adhering to~\eqref{eq:condition on k}. Then~$H$ is thetyspectral with factor~$|G:H|$.
\end{proposition}

Note that all subgroups arising in this way -- i.e.\ that are subject to the conditions above -- are of index~$q^{\dim_{\Op} \SL_3(\Op) n}$ in~$\SL_3^m(\Op)$, for some positive integer~$n \in \N$. At the same time, the proposition only deals with a very special kind of subgroup.

\begin{problem}
	Find an open subgroup of $\SL_3^m(\Op)$ that is not thetyspectral.
\end{problem}

Clearly, the method described above can be used in the hope of obtaining thetyspectral subgroups in other uniformly potent pro-$p$ groups. An example of a non-the\-ty\-spec\-tral subgroup in $\SL_2^1(\Op)\ltimes \Op^2$ is given in \cref{sec:a_non_thetyspectral_subgroup}.

\subsection{Lemmata on potency}

In order to give a proof of \cref{thm:main}, we need to establish that certain subgroups of potent pro-$p$ groups are again potent to make use of \cref{thm:subgroup formula}.

\begin{lemma}\label{lem:potent intersection}
    Let~$G$ be a uniformly potent pro-$p$~group. Let~$H$ and~$K \leq G$ be two closed potent subgroups. Then~$H \cap K$ is potent.
\end{lemma}

\begin{proof}
    It is a general fact that $\gamma_{p-1}(H \cap K) \leq \gamma_{p-1}(H) \cap \gamma_{p-1}(K)$. Since $H$ and $K$ are potent, this implies $\gamma_{p-1}(H \cap K) \leq H^p \cap K^p$. Potent pro-$p$ groups have unique~$p$\textsuperscript{th} roots, whence $H^p \cap K^p = (H \cap K)^p$. This concludes the proof.
\end{proof}

\begin{lemma}\label{lem:potent point stabs}
    Let~$\mathfrak{g}$ be an $\Op$-Lie lattice, let $m$ be soundly permissible for $\mathfrak{g}$, and let~$\sigma \colon \mathfrak{g} \to \mathfrak{p}^m\mathfrak{gl}_n$ be a faithful $\Op$-Lie homomorphism. Let~$\chi$ be an irreducible complex representation of the group $\Op^n$ and let $G = \exp \mathfrak{p}^m \mathfrak{g}$. The stabiliser $\Stab_G(\chi)$ of~$\chi$ under the action of $G$ on~$\Irr(\Op^n)$ via~$\exp \circ \sigma \circ \log$ is uniformly potent.
\end{lemma}

\begin{proof}
    Put $S = \Stab_{\GL_n^m(\Op)}(\chi)$. We will show that $S$ is uniformly powerful and thus uniformly potent in particular; whence by \cref{lem:potent intersection} also $\Stab_G(\chi) = S \cap G$ is uniformly potent.

    The group $\Irr(\Op^n)$ is isomorphic to the $n$\textsuperscript{th} direct power of the Prüfer group of all roots of unity of $p$-power order $\mathbb{C}_{p^{\infty}}^n$ and every $\GL_n(\Op)$-orbit contains an element of the form $(z, 1, \dots, 1)$, for some $z \in \mathbb{C}_{p^\infty}$. Replacing $\sigma$ with $g \mapsto (g^\sigma)^h$ and the stabiliser with its conjugate under $h$ for some $h \in \GL_n(\Op)$ if necessary, we can assume that $\chi$ is equivalent to an element of the given form. Let $p^k = q$ be the order of $z$. Clearly, the $k$\textsuperscript{th} principal congruence subgroup stabilises $(z, 1, \dots, 1)$; thus $S$ is a finite index subgroup. Furthermore, $S = \GL_n^m(\Op)$ is powerful in case $k < m$. Otherwise, a quick computation shows that $S$ is the set of elements in $\GL_n^m(\Op)$ of the form
    $$
        s = \begin{pmatrix}
            1+\pi^k a & \pi^k x\\
            \pi^m y^\intercal & M,
        \end{pmatrix},
    $$
    for $a \in \Op$, $x, y \in \Op^{n-1}$, and $M \in \GL_{n-1}^m(\Op)$. One sees that
    $$
        S = \GA_{n-1}^m(\Op) \cdot T^k.
    $$
    Here, the $m$\textsuperscript{th} principal congruence subgroups of general affine group $\GA_{n-1}^m(\Op) = \GL_{n-1}^m(\Op) \ltimes \pi^m\Op^{n-1}$ is represented by matrices of the form
    $$
        \begin{pmatrix}
            1 & 0\\
            \pi^m y^\intercal & M
        \end{pmatrix},
    $$
    with $y \in \Op^{n-1}$ and $M \in \GL_{n-1}^m(\Op)$, and $T^k$ denotes the abelian group of matrices of the form 
    $$
        \begin{pmatrix}
            1 + \pi^k \Op & \pi^k x\\
            0 & \operatorname{Id}_{(n-1)\times(n-1)}
        \end{pmatrix},
    $$
    with $x \in \Op^{n-1}$. The $m$\textsuperscript{th} principal congruence subgroup of the general affine group is a semi-direct product of a powerful group with an abelian normal subgroup, whence it is powerful by \cref{lem:semidirect potent}. Let $g, h \in \GA_{n-1}^m(\Op)$ and $a, b \in T_k$, then
    $$
        [ga, hb] = [g, b]^a [g, h]^a [a, h]^b.
    $$
    Consider the case of an odd prime $p$. Since $\GA_{n-1}^m(\Op)$ is powerful, the middle term is contained in $\GA^m_{n-1}(\Op)^p \leq S^p$. The other two terms are contained in $[\GL_n^{m}(\Op), \GL_n^{k}(\Op)] \leq \GL_n^{k+m}(\Op)$, which is contained in $\GL_n^{k+1}(\Op) = \GL_n^k(\Op)^{p} \leq S^{p}$. Thus, every commutator in $S$ is contained in $S^p$, and $S$ is powerful. For $p = 2$, the proof works analogously.
\end{proof}

\subsection{Product formula for representation zeta functions of semi-direct products}

We are now able to prove our main result.

\begin{proof}[Proof of \cref{thm:main}]
    Let~$\mathcal{X}$ be a set of representatives of the~$H$-orbits of~$\Irr(\Op^n)$. By \cref{prop:reps of sd prod}, all irreducible representations of~$G$ are parametrized by the representatives~$\chi\in \mathcal{X}$ and the irreducible representations of the stabilisers~$S_\chi:=\Stab_H(\chi)$ of~$\chi\in \mathcal{X}$, and are of the form
    \begin{equation*}
        \Ind_{S_\chi\ltimes V }^{G}(\Inf_{S_\chi}^{ S_\chi \ltimes V }(\tau) \otimes \Ext_{V}^{ S_\chi \ltimes V}(\chi))
    \end{equation*}
    for some $\tau \in \Irr(S_\chi)$. The dimension of such a representation is given by the product~$|H : S_\chi| \cdot \dim(\tau)$. Thus
    \[
        r_n(G) = |\Irr_n(G)| = \sum_{\substack{a, b \in \N\\ ab = n}} \sum_{\substack{\chi \in \mathcal{X}\\ |H : S_\chi| = a}} r_b(S_\chi).
    \]
    By \cref{lem:potent point stabs} the group~$S_\chi$ is a uniformly potent pro-$p$ group, hence \cref{cor:subgroups} implies~$r_b(S_\chi) = |H : S_\chi| \cdot r_b(H)$.
    
    Now to every~$\chi \in \mathcal{X}$ we may associate the~${|H:S_\chi|}$\nobreakdash- dimensional representation $\Ind_{ S_\chi\ltimes \Op^n}^G(\Ext_{\Op^n}^{ S_\chi\ltimes \Op^n} (\chi))$. In view of \cref{prop:consts of idrep}, these are precisely the~\mbox{$|H:S_\chi|$\nobreakdash-}di\-men\-sion\-al irreducible constituents of~$\Ind_H^G(\idrep_H)$, hence
    \begin{align*}
        r_n(G) = \sum_{\substack{a, b \in \N\\ ab = n}} \sum_{\substack{\chi \in \mathcal{X}\\ |H : S_\chi| = a}} a \cdot r_b(H) = \sum_{\substack{a, b \in \N\\ ab = n}} r_a(G,H) \cdot a \cdot r_b(H).
    \end{align*}
    We see that the numbers~$r_n(G)$ result from the Dirichlet convolution of the arithmetic sequences~$a \cdot r_a(G,H)$, with~$a\in\mathbb{N}$ and~$r_b(H)$, with~$b\in \mathbb{N}$. The factor~$a$ corresponds to a shift in the Dirichlet generating function of the sequence~$r_a(G,H)$, i.e.~$\sum_{a \in \N} r_a(G,H) \cdot a \cdot a^{-s} = \zeta_{H}^G(s-1)$. Since the generating function of a Dirichlet convolution is the product of the corresponding generating functions, this concludes the proof.
\end{proof}

\section{Uniformly potent pro-\emph{p} groups of large representation growth}
\label{sec:uniformly_potent_pro_p_groups_of_large_representation_growth}

\subsection*{General strategy}
\label{sub:general_strategy}
Building on \cref{thm:main}, we now set out to prove \cref{thm:diagonal on natural module} and \cref{thm:diagonal on symmetric square}. Let $m$ be soundly permissible for~$\mathfrak{sl}_2(\Op)$ and let ${\varphi \colon \SL_2^m(\Op) \to \GL_n^m(\Op)}$ be a faithful $\Op$-representation of $\SL_2^m(\Op)$ such that the semi-direct product $G=\SL_2^m(\Op)\ltimes V\cong\SL_2^m(\Op) \ltimes_\varphi \Op^n$ is FAb, where $V$ is the $\Op$-module defined by $\varphi$. We have to compute relative representation zeta functions of the form~$\zeta_{\SL_2^m(\Op)}^{G}(s)$. Let $\mathfrak{g} = \mathfrak{h} \oplus \mathfrak{v}$ be the $\Op$-Lie lattice such that $G=\exp(\mathfrak{p}^m\mathfrak{g})$ as in~\cref{thm:rel zeta integral}, where~$V = \exp(\mathfrak{p}^m\mathfrak{v})$ and~$H = \SL_2^m(\Op)=\exp(\mathfrak{p}^m\mathfrak{h})$. 
Using the terminology of~\cref{prop:integral formula-semidirect case}, we choose a basis~$\mathfrak{X}$ of~$\mathfrak{v}$ and a basis~$\mathfrak{Y}$ of~$\mathfrak{h}$. Additionally, we put~$\mathfrak{Z} = \mathfrak{X} \cup \mathfrak{Y}$. Given~$n \in \N$ and $i, j \in [n]$, the $n$-by-$n$ matrix over~$\Op$ with a single non-zero entry equal to~$1$ at the position~$(i,j)$ is denoted~$E^{n}_{i,j}$. We will make use of these notational conventions throughout this section and specialise as needed.

Since we have to specifically choose a basis for our computations, we represent the group $G$ as follows. As a basis for $\mathfrak{sl}_2(\Op)$, we fix the elements
\[
	\underline{h} = \begin{pmatrix}
		1 & 0 \\ 0 & -1
	\end{pmatrix}, \quad
	\underline{e} = \begin{pmatrix}
		0 & 1 \\ 0 & 0
	\end{pmatrix}, \quad \text{and} \quad
	\underline{f} = \begin{pmatrix}
		0 & 0 \\ 1 & 0
	\end{pmatrix},
\]
and we put $h = \underline{h}^\varphi, e = \underline{e}^\varphi$ and $f = \underline{f}^\varphi$.

\subsection*{Proofs of \cref{thm:diagonal on natural module} and \cref{thm:diagonal on symmetric square}} 
\label{sub:proofs}

\begin{proof}[Proof of \cref{thm:diagonal on natural module}]
	Here, $n \in \N$ and $\mathfrak{v} = \Op^{2n}$. The semi-direct product $G_n$ is due to the diagonal of the natural action of $\SL_2^m(\Op)$, which naturally extends to an action of~$\mathfrak{h}$ on~$\mathfrak{v}$ via~$x^\varphi = \diag(x, \dots, x)$ for all $x \in \mathfrak{h}$, and~$m$ is soundly permissible for~$\mathfrak{sl}_2(\Op)$. The natural basis for $\mathfrak{v}$ is given by $\mathfrak{X} = \{u_i \mid i \in [n]\} \cup \{v_i \mid i \in [n]\}$ with $u_i = E_{{2i-1, 2n+1}}^{2n+1}$ and $v_i = E_{{2i, 2n+1}}^{2n+1}$ for $i \in [n]$. It is easily verified that the natural action on $\Op^2 = \langle u_1, v_1\rangle$ is given by ${h.u_1 = u_1}, {h.v_1 = -v_1,} {e.u_1 = 0,} {e.v_1 = u_1,}{f.u_1 = v_1,}{ f.v_1 = 0}$, whence we need to consider the minors of the matrix
	\[
		\Com(\mathfrak{X}, \mathfrak{Y}) = \begin{pmatrix}
			u_1 & -v_1 &  u_2 & -v_2 & \cdots &  u_{n} & -v_{n} \\
			v_1 &  0   &  v_2 &  0   & \cdots &  v_{n} &  0      \\
			0   &  u_1 &  0   &  v_2 & \cdots &  0     &  u_{n}
		\end{pmatrix}^\intercal,
	\]
	which are, up to sign, given by
    \begin{gather*}
        \{1\} \,\cup\,
		\{a\mid a\in \mathfrak{X} \} \,\cup\,
		\{ab\mid a,b \in \mathfrak{X} \} \,\cup\,
		\{\det(i,j) \mid i, j \in [n]\}\\
		\cup\,\{a\det(i,j) \mid a \in \mathfrak{X}, i, j \in [n]\},
    \end{gather*}
    where we put~$\det(i, j) = u_{i} v_{j} - u_{j} v_{i}$. The maximal norm of the given elements is never exclusively attained at one of the monomials of degree~$1$. 
	Similarly, it is sufficient to consider only unmixed monomials of degree~$2$ and the minors of the form $a\det(i,j)$. Consequently, put
    \[
        P_1 = \{ x^2 \mid x \in \mathfrak{X}\}, \quad P_2 = \{a \det(i,j) \mid a \in \mathfrak{X}, i, j \in [n], i < j\};
    \]
    however, also the following set will be useful later,
    \[
       P_2' = \{\det(i,j) \mid i, j \in [n], i < j\}.
    \]
	We now partition the area of integration into convenient pieces. First, notice that the maximum $||\{1\} \cup P_1 \cup P_2||_{\p}$ is exclusively attained by $|1|_{\p}$ precisely on $\Op^{2n}$, whence
	\[
		\zeta_{\SL_2^m(\Op)}^{G_n}(s) = q^{2mn}\left(\int_{\Op^{2n}\mathfrak{X}} |1|_{\p}^{-s-1}\mathrm{d}\mu(w) + \int_{K^{2n}\mathfrak{X} \smallsetminus \Op^{2n}\mathfrak{X}} \|(P_1 \cup P_2)(w)\|_{\p}^{-s-1}\mathrm{d}\mu(w)\right).
	\]
	The left integral evaluates to $1$ by our choice of normalisation, while the area of integration of the right integral may be further disassembled into the sets $F_j = \pi^{-j}(\Op^{2n}\mathfrak{X} \smallsetminus \pi\Op^{2n}\mathfrak{X})$ for $j \in \N$. Here, the minimal valuation of the variables is constant and equal to $j$, whence we may normalise by a change of variables $w \mapsto \pi^j w$ and obtain
	\begin{align*}
		\int_{F_j} \|(P_1 \cup P_2)(w) \|_{\p}^{-s-1}\mathrm{d}\mu(w)
		&=
		|\pi^{2nj}|_{\p}\int_{F_0} \|\pi^{-2j}(P_1 \cup \pi^{-j}P_2)(w) \|_{\p}^{-s-1}\mathrm{d}\mu(w)\\
		&=
		(q^{n-1}t)^{2j}A(j),
	\end{align*}
	where 
    \begin{align*}
        A(j) =& \int_{F_0}\| (P_1 \cup \pi^{-j}P_2)(w)\|_{\p}^{-s-1} \mathrm{d}\mu(w) = \int_{F_0} \| \{ 1 \} \cup \pi^{-j} P_2'(w) \|_{\p}^{-s-1} \,\mathrm{d}\mu(w).
    \end{align*}
	Here, the second equality stems from the fact that on $F_0$ there is at least one invertible variable~$x$, whence the maximum attained on the set~$P_1$ is always equal to~$|x|_{\p} = |1|_{\p}$, and, by~$|a\det(i,j)|_{\p} \leq |x\det(i,j)|_{\p} = |\det(i,j)|_{\p}$, the maximum on $P_2$ equals the maximum on $P_2'$.
	
	We further partition the area of integration as follows. Put
	\[
		I_k = \pi\Op^{k-1}\mathfrak{X} \times \Op\mathfrak{X}\setminus \pi\Op\mathfrak{X} \times \Op^{2n-k}\mathfrak{X}
	\]
	for $k \in [2n]$. Naturally, $F_0$ is the union of all $I_k$ ranging over $k \in [2n]$. Assume that $k = 2i-1$ is odd. On $I_k$ substitute $u_{i}$ with $z_i = \det(i, j)$ for $j \in [n]\smallsetminus\{i\}$. For convenience, put $z_{i} = u_i$, for the substitution in $p$-adic integrals, see \cite[Proposition 7.4.1]{Igu00}. Clearly, the determinant of the corresponding Jacobian has norm~$1$. Furthermore, the set~$I_k$ is invariant under the transformation. The polynomial~$\det(l, j)$ for~$l, j \in [n]\smallsetminus\{i\}$, with~$l \neq j$, may be rewritten as
    \[
        \det(l, j) = u_{l} v_{j} - u_{j} v_{l} = u_i^{-1}(z_{j} v_{j} - z_i v_{l}).
    \]
	Since $v_j \in \Op$ for all $j \in [n]$, we find
	\[
		|\det(l,j)|_{\p} \leq \| \{z_l, z_j\} \|_{\p}.
	\]
	Put $\mathfrak{Z} = \{z_j \mid j \in [n] \smallsetminus \{i\}\}$ and $\mathfrak{X}' = \mathfrak{Z} \cup \{z_i\} \cup \{v_{j} \mid  j \in [n] \}$. The case of even~$k$ works analogously by transforming the variables $u_j$ etc. It remains to resolve the integral over $I_k$ for every $k \in [2n]$, i.e.\
	\[
		\mu(\pi\Op^{\lfloor k/2 \rfloor} \times \Op\smallsetminus \pi\Op) \int_{(\pi\Op^{\lceil k/2\rceil-1} \times \Op^{n-\lceil k/2 \rceil})\mathfrak{Z}} \|\{1\} \cup \pi^{-j}\mathfrak{Z}(w)\|_{\p} \mathrm{d}\mu(w).
	\]
	For odd $k$, we see that the second integral is equal to the one for $k+1$, whence we may concentrate on
	\[
		B_n(j,i) := \int_{(\pi\Op^{i-1} \times \Op^{n-i})\mathfrak{Z}} \|\{1\} \cup \pi^{-j}\mathfrak{Z}(w)\|_{\p} \mathrm{d}\mu(w)
	\]
	for $i \in [n]$, as
	\[
		A(j) = \sum_{i = 1}^{n} (1-q^{-1})(q^{-i} + q^{1-i}) B_n(j,i) = \sum_{i = 1}^{n} (1-q^{-2})q^{1-i} B_n(j,i).
	\]
    To evaluate $B_n(j,i)$, we integrate $\| \{1\} \cup \pi^{-j} \mathfrak{Z}(w) \|_{\p}^{-s-1}$ over $\Op^{n-1}\mathfrak{Z}$ and over the complement of $(\pi\Op^{i-1} \times \Op^{n-i}\mathfrak{Z})$ individually and take the difference; finding
    \begin{align*}
        \int_{\pi^j\Op^{n-1}\mathfrak{Z}} \|1\|_{\p}^{-s-1} \,\mathrm{d}\mu(w) &+ \sum_{r = 0}^{j-1} \int_{\pi^r(\Op^{n-1}\setminus \pi\Op^{n-1})\mathfrak{Z}} \| \pi^{-j} \mathfrak{Z}(w) \|_{\p}^{-s-1} \,\mathrm{d}\mu(w)\\
        &=q^{-j(n-1)}+\left( 1-q^{-(n-1)}\right)(q^{-1}t)^j \sum_{r=0}^{j-1} (q^{(2-n)}t^{-1})^r\\
        &=q^{-j(n-1)}\left(1+\left( 1-q^{-(n-1)}\right)\frac{1-(q^{(n-2)}t)^j}{1- q^{(n-2)}t}q^{n-2}t\right),
    \end{align*}
    and, more directly,
    \begin{align*}
        \int_{((\Op^{i-1}\setminus \pi\Op^{i-1} )\times \Op^{n-i})\mathfrak{Z}} \| \{1\} \cup \pi^{-j} \mathfrak{Z}(w) \|_{\p}^{-s-1} \,\mathrm{d}\mu(w)
        = (1-q^{1-i}) (q^{-1}t)^j.
    \end{align*}
    Thus we have 
    \[
        B_n(j,i) = q^{-j(n-1)}\left(1+\left( 1-q^{-(n-1)}\right)\frac{1-q^{j(n-2)}t^j}{1- q^{(n-2)}t}q^{n-2}t\right)-(1-q^{1-i}) (q^{-1}t)^j.
    \]
    Overall we find
    \begin{align*}
        A(j) &= q^{-j(n-1)} \left( \frac{(1-q^{-n})(1-q^{-2})}{1-q^{n-2}t}\right)\\
        &\quad \quad\quad\quad\quad -(q^{-1}t)^j \left( q^{-1}\frac{(1-q^{-n})(1-q^{1-n})(1+q^{n-1}t)}{1-q^{n-2}t}\right).
    \end{align*}
	Returning to the integral over $K^{2n}\smallsetminus\Op^{2n}$, we find its value to as
    \begin{align*}
        \sum_{j = 1}^{\infty} (q^{(n-1)}t)^{2j} A(j)
        &= (1-q^{-n})\frac{q^2+q^3+(q^{n+1}-q^2)t-(1+q^n)q^nt^3}{(1-q^{(n-1)}t^{2})(1-q^{(2n-3)}t^{3})}q^{n-4}t^2.
    \end{align*}
    In consequence, the full representation zeta function of~$G_n$ equals
    \[
        \zeta_{G_n}(s) = q^{2nm}\frac{(1-(1+q-q^{n})t^2-(1-q^{n-1}+q^{n})qt^3+q^{n+1}t^5)}{(1-q^{(n+1)}t^{2})(1-q^{2n}t^{3})} \zeta_{\SL_2^m(\Op)(s)}.\qedhere
    \]
\end{proof}

\begin{proof}[Proof of \cref{thm:diagonal on symmetric square}]
	We proceed similarly to the previous proof, whence we describe the computation more concisely. In the present case, $n \in \N$, $V_n = \Sym^2(\mathfrak{p}^m\Op^2)^n \cong \Op^{3n}$, $G_n=\SL_2^m(\Op)\ltimes V_n$, and $m$ soundly permissible for $\mathfrak{sl}_2(\Op)$. The action of $\SL_2^m(\Op)$ on the symmetric square yields the following blocks for the block diagonal matrices $h = \diag(h_1, \dots, h_1), e = \diag(e_1, \dots, e_1)$ and $f = \diag(f_1, \dots, f_1)$ representing $\mathfrak{h}$, 
	\[
		h_1 = \begin{pmatrix}
			2 & 0 & 0 \\ 0 & 0 & 0 \\ 0 & 0 & -2
		\end{pmatrix}, \quad
		e_1 = \begin{pmatrix}
			0 & 2 & 0 \\ 0 & 0 & 1 \\ 0 & 0 & 0
		\end{pmatrix}, \quad \text{and}\quad
		f_1 = \begin{pmatrix}
			0 & 0 & 0 \\ 1 & 0 & 0 \\ 0 & 2 & 0
		\end{pmatrix}.
	\]
	It is useful to write $u_i = E_{3i-2, 3n+1}^{3n+1}, v_i = E_{3i-1, 3n+1}^{3n+1}$ and $w_i = E_{3i, 3n+1}^{3n+1}$, with $i \in [n]$, for the $3n$ elements of $\mathfrak{X}$. We compute the block row matrix
	\[
		\Com(\mathfrak{g}, \mathfrak{X}, \mathfrak{Y}) = \begin{pmatrix}
		2u_1 & 0   & -2w_1 & \cdots & 2u_n & 0   & -2w_n \\
		2v_1 & w_1 &  0    & \cdots & 2v_n & w_n &  0    \\
		0    & u_1 &  2v_1 & \cdots & 0    & u_n &  2v_n
		\end{pmatrix}^\intercal.
	\]
	Up to sign and powers of $2$, which, since $p$ is odd by assumption, do not influence the norm, the set of minors is given by
    \begin{gather*}
   		\{ab \mid a,b \in \mathfrak{X} \cup \{1\}\} \cup
		\{ u_iv_jw_l-u_lv_iw_j \mid i, j, k \in [n] \}\\ \cup\,
		\{a \det(i,j; r, s) \mid a \in \mathfrak{X} \cup \{1\}, i, j \in [n], r, s \in \{u,v,w\}\},
    \end{gather*}
    where~$\det(i,j; r, s) = r_is_j - r_js_i$. 
    Exactly as in the previous proof, the monomials of degree one, the mixed monomials, and the polynomials of the form $\det(i,j; r,s)$ without factor do not play any role in the computation of the representation zeta function. Write $P$ for the set of minors above, excluding the unmixed monomials of degree~$2$. Partitioning $K^{3n}$ into
    \[
        K^{3n}=\Op^{3n}\cup\bigcup_{j=1}^\infty \pi^{-j}(\Op^{3n})^\times,
    \] 
    and transforming the latter parts by multiplication with $\pi^j$, we obtain
    \begin{align*}
        \zeta_{\SL_2^m(\Op)}^{G_n}(s)&=
        q^{3mn} \left( 1 + \sum_{j=1}^{\infty} \int_{\pi^{-j}(\Op^{3n})^\times\mathfrak{X}}\|(\{x^2 \mid x \in \mathfrak{X}\}\cup P)(w)\|_{\p}^{-s-1} \mathrm{d}\mu(w)\right)\\
        &=q^{3mn} \left( 1 + \sum_{j=1}^{\infty} q^{j(3n-2)}t^{2j} C(j)\right),
    \end{align*}
    with
    \begin{align*}
        C(j) = &\int_{(\Op^{3n})^\times\mathfrak{X}}\|(\mathfrak{X}^2\cup \pi^{-j}P)(w)\|_{\p}^{-s-1} \mathrm{d}\mu(w)\\
        =&\int_{(\Op^{3n})^\times\mathfrak{X}}\|\{1\}\cup \pi^{-j}P(w) \|_{\p}^{-s-1} \mathrm{d}\mu(w).
    \end{align*}
    We further partition the domain of integration in to~$n$ parts by isolating the first triple of variables~$(u_{k},v_{k},w_{k})$ for~$k \in [n]$ containing an invertible, i.e.\ into
    \[
        I_k = (\pi\Op^{3(k-1)}
        \times(\Op^{3})^\times\times \Op^{3(n-k)})\mathfrak{X}
    \]
    for $k \in [n]$. Consider the integral over $I_k$ and assume without loss of generality that~$u_{k}$ is invertible. For every~$a \in \mathfrak{X}$, every $i,j \in [n]$ with~$i \neq j$ and every $r,s \in \{u,v,w\}$ with~$i \neq j$, we have for all $w \in I_k$
    \begin{align*}
		| a \det(i,j; r, s)(w)|_{\p} \leq
		| \det(i,j; r,s)(w) |_{\p}.
    \end{align*}
    Perform the (linear) change of variables $r_i \mapsto \tilde{r}_i = \det(k, i; u, r)$ for $i \in [n]\smallsetminus\{k\}$ and $r \in \{v, w\}$. The area of integration is invariant and the determinant of the Jacobian is invertible. In rewriting our integrand, we find it to be the minimum of the norms of the functions
    \begin{align*}
        &\tilde{v}_i,\tilde{w}_i, \det(v_k,\tilde{w}_j),\det(u_i,\tilde{v}_j),\det(u_i,\tilde{w}_j),\det(\tilde{v}_i,\tilde{w}_j), \\
        &u_i\tilde{v}_j\tilde{w}_l-u_l\tilde{v}_i\tilde{w}_j+u_iu_j\det(v_k,\tilde{w}_l)-u_iu_l\det(v_k,\tilde{w}_j),
    \end{align*}
    where~$i,j,l\in[n]\setminus\{k\}$ are pairwise different. A comparison of norms shows that the only relevant polynomials are the monomials~$\tilde{v_i}$ and~$\tilde{w}_i$ with~$i\in [n]\setminus\{k\}$. Consequently, we may consider a new set of variables~$\mathfrak{Z}=\{z_1,\dots,z_{2(n-1)}\}$, yielding 
    \begin{align*}
		C(j) &= (1-\frac{1}{q^3}) \sum_{k=1}^n\frac{1}{q^{k-1}}\int_{(\pi\Op^{2(k-1)}\times \Op^{2(n-k))}\mathfrak{Z}} \|\{1\}\cup \pi^{-j}\mathfrak{Z}(w)\|_{\p}^{-s-1} \mathrm{d}\mu(w)\\
        &=(1-\frac{1}{q^3}) \sum_{k=0}^{n-1}\frac{1}{q^{k}}B_{2n-1}(j,2k+1),
    \end{align*}
    where~$B_n(j,i)$ is the expression defined and evaluated in the proof of \cref{thm:diagonal on natural module}. Thus, by routine calculations we get
    \begin{align*}
        C(j) =& q^{-2j(n-1)}\frac{(1-q^{-3})(1-q^{-n})(1-q^{-1}t)}{(1-q^{-1})(1 -q^{2n-3}t)}\\ &-q^{-j}t^j(1-q^{-n})(1-q^{1-n})\left(q^{-1}\frac{(1+q^{-1}+q^{-n})+(1+q+q^{n})q^{n-2}t}{1 -q^{2n-3}t}\right),
    \end{align*}
    leading to
	\[
		\zeta_{\SL_2^m(\Op)}^{G_{n}}(s) = q^{3nm-2}\frac{(q-t)(1-t)(t+q^{n}t^2+q(1+t)(1+q^{n-2}t^2))}{(1-q^{3n-3}t^3)(1-q^nt^2)},
	\]
	and the claimed expression for the full representation zeta function.
	
	Finally, we consider the representation zeta function of the semi-direct product~$G= \SL_2^m(\Z_2) \ltimes \Sym^2(\Z_2^2)$ for~$m\in \mathbb{N}_{\geqslant 2}$. Following along the computation in the odd case, the divergence in the computation follows from the need to keep track of the powers of $2$ dividing the minors of
	\[
		\begin{pmatrix}
			2u & 0 & -2w \\
			2v & w & 0   \\
			0  & u & 2v
		\end{pmatrix};
	\]
	the (relevant) minors are given by $\{ 1, 2u^2, 4v^2, 2w^2\}$, here one uses the fact that $||1, u, 2u^2||_2 = |1, 2u^2|_2$. The transformation $v \mapsto 2v$ yields the integral
	\begin{align*}
		\zeta_{\SL_2^m(\Z_2)}^{G}(s)&=
		2^{3m+1} \left( 1 + \sum_{j=1}^{\infty} \int_{2^{-j}(\Z_2^{3})^\times \mathfrak{X}}\|(\{1,2u^2,v^2,2w^2\})(w)\|_{2}^{-s-1} \mathrm{d}\mu(w)\right)\\
		&= 2^{3m+1}\left(1 + \sum_{j=1}^\infty 2^{3j}(1-2^{-2})2^{(2j-1)(-1-s)} + (1-2^{-1})2^{2j(-1-s)-2}\right)\\
		&= 2^{3m-1}\frac{(2-t)(2^3t+(2-t))}{1-2t^2},
	\end{align*}
	whence the zeta function of~$G$ is given by the described rational function.

    The discrepancy between the odd and the even case is due to the presence of the prime $2$ among the structure constants of the considered Lie lattice. More generally, for a Lie lattice, the structure constants involve only finitely many integers, so discrepancies can occur at only finitely many primes.
\end{proof}


\subsection*{Further discussion of the results} 
\label{sub:further_discussion_of_the_results}

\begin{enumerate}
	\item One readily observes that the computation of the irreducible cases $n = 1$ in the last two theorems do not require any argument past the standard decomposition of $\Q_p^d$, since there are no non-trivial $3\times3$-minors of the matrices~$A$.
	\item For any algebraic counting problem giving rise to a well-behaved, uniform (i.e.\ only polynomially dependent on the prime $p$) zeta function $\zeta(s) = W(p, p^{-s})$, one may consider its \emph{reduced zeta function} $\zeta^{\mathrm{red}}(s) := W(1, T)$, see \cite{Evs09}. From our formulae one derives that
	\[
	\zeta^{\mathrm{red}}_{G}(s) = \zeta^{\mathrm{red}}_{\SL_2^m(\Op)}(s) = 1
	\]
	for $G = \SL_2^m(\Op)\ltimes\Op^{2n}$ and $G = \SL_2^m(\Op)\ltimes\Sym^2(\Op^{2})^n$, for any $n$.
	\item The \emph{topological zeta functions} related to our examples, see \cite{Ross15} for a definition, are
	\[
	\zeta_{G}^{\mathrm{top}}(s) =
	\frac{s(s+2)(6s-3n-1)}{(s-1)(2s-n-1)(3s-2n)}
	\]
	for $G = \SL_2^m(\Op)\ltimes\Op^{2n}$ and
	\[
	\zeta_{G}^{\mathrm{top}}(s) =
	\frac{s(s+2)(3s-n-1)}{(2s-n-2)(s-n)}
	\]
    for $G = \SL_2^m(\Op)\ltimes\Sym^2(\Op^{2})^n$.
	\item In both cases, the variety defined by the degree-$3$ minors over $\F_q$ can be shown to be smooth away from the singular origin. As a consequence, an alternative computation may be carried out by the determination of the size of the zero locus, making use of Igusa's stationary phase formula, c.f.\  \cite{Igu94}.
	\item The relative simplicity of the computations is deceptive: already the natural next example $\Sym^3(\Op^2)$ is highly difficult to compute, reflecting the fact that the zero locus is everywhere singular. Similarly, the mixed representation $\Sym^2(\Op^2) \times \Op^2$ requires a much more involved computation due the non-triviality of the singular locus.
\end{enumerate}


\section{A non-thetyspectral subgroup} 
\label{sec:a_non_thetyspectral_subgroup}

Finally, as a counterpart to \cref{cor:subgroups} and \cref{prop:thethys in sl3}, we construct a subgroup of $G = \SL_2^1(\Op) \ltimes \Op^2$ that is not thetyspectral, where, for the sake of simplicity, we consider $\Op$ to be an unramified extension of~$\mathbb{Z}_p$ and residue class field of cardinality~$q$, where $q$ is a power of an odd prime $p$. In the general case (but still assuming~$p$ odd, otherwise we would also need to take into account the structure constants), we could perform similar computations with an additional parameter $m$ to ensure soundly permissibility of the $\Op$-Lie lattice.
Let $\chi_k$ be the irreducible representation of $\Op^2 = \langle u, v \rangle$ that maps $u$ to a primitive $p^k$\textsuperscript{th} root of unity, with $k \in \N$, and~$v$ to~$1$. The stabiliser~$S_k = \Stab_{\SL_2^1(\Op)}(\chi_k)$ is, according to \cref{lem:potent point stabs}, a uniformly potent subgroup of $\SL_2^1(\Op)$. Put $H_k = S_k \ltimes \Op^2$, which is a uniformly potent group by \cref{lem:semidirect potent}.

Consider the~$\Op$-Lie~lattice~$\mathfrak{g}$ generated by the elements~${h}, {e}, {f}, {u},$ and~$ {v}$ as in the previous section; such that $G = \exp(p\mathfrak{g})$. Let $\mathfrak{s}_k$ be the~$\Op$-Lie sublattice of~$\mathfrak{g}$ generated by the elements~$p^{k}{h}, {e}$, and~$p^{k} {f}$ and let $\mathfrak{h}_k$ be generated by the elements~$p^{k} {h}, {e}, p^{k} {f}, {u}$, and~${v}$. Then $S_k = \exp(p \mathfrak{s}_k)$ and $H_k = \exp(p\mathfrak{h}_k)$. For simplicity, we assume that~$p$ is odd.

\begin{proposition}\label{prop:non-thety}
	For every $k \neq 0$, the finite-index subgroup $H_k \leq \SL_2^1(\Op) \ltimes \Op^2$ is not thetyspectral and its representation zeta function is given by
	\[
		\zeta_{H_k}(s) = q^{2k+2}\frac{(1-t)((1-(qt)^{k+1}) + qt^2(1-(qt)^{k-1})}{(1-qt)^2(1+qt)}
		\cdot \zeta_{\SL_2^1(\Op)}(s).
	\]
\end{proposition}

\begin{proof}
	By \cref{thm:main} and \cref{cor:subgroups}, we find
    \[
        \zeta_{H_k}(s) = \zeta_{S_k}^{H_k}(s-1) \cdot \zeta_{S_k}(s) = |\SL_2^1(\Op):H_k| \cdot \zeta_{S_k}^{H_k}(s-1) \cdot \zeta_{\SL_2^1(\Op)}(s).
    \]
    In view of the description of its Lie lattice, the index of~$S_k$ in~$\SL_2^1(\Op)$ is easily seen to be~$q^{k}$. It remains to calculate $\zeta_{S_k}^{H_k}$ following the procedure laid out in the previous section. The partial commutator matrix is given by
    \[
        \begin{pmatrix}
			p^{k}u & -p^{k}v 	\\
				 v &  0 		\\
				 0 &  p^{k}u
        \end{pmatrix}
    \]
    and the set of minors (up to sign) is
    \[
        \operatorname{Min}(A) = \left\{ 1, p^{k}u, v, p^{k} v, p^{k}uv, p^{2k}u^2, p^{k}v^2\right\}.
    \]
    After the transformation~$u \mapsto p^{k} u$, the integral describing $\zeta_{S_k}^{H_k}$ is
    \[
        \zeta_{S_k}^{H_k}(s) = q^{k+2} \int_{\Q_p^2 \mathfrak{X}} \lVert (\{1, u, v, uv, u^2, p^{k} v^2\})(w) \rVert_{\p}^{-1-s} \mathrm{d}\mu(w).
    \]
    The polynomial $u$ is irrelevant. Comparing the valuations of the remaining polynomials, we determine that the maximum of the norm is reached by~$1$ in the area~$\Op^2$, by~$u^2$ within~$\bigcup_{m \in \N_0} p^{-m}\Op^\times \times p^{-m}\Op$, by~$uv$ within~$\bigcup_{m \in \N_0} p^{-m}\Op^\times \times p^{-m-k}\Op \smallsetminus p^{-m+1}\Op$, by~$p^k v^2$ within~$\bigcup_{m \in \N_0} p^{-m}\Op \times p^{-m-k}\Op^\times$, and by~$v$ in the area~$\Op \times (p^{-k}\Op \smallsetminus p\Op)$. 

    These areas overlap. We cut up~$K^2$ into areas on which a fixed polynomial has maximal norm, see~\cref{fig:divisions} for comparison, such that the full integral is the sum of the integrals over the parts.
    \begin{figure}[ht]
	\centering
	\begin{tikzpicture}
		\draw[thick,-Triangle] (-4,-2) -- (3,-2);
			 \draw (0,-1.9) -- (0,-2.2) node[anchor=north] {0};
			 \draw (-2,-1.9) -- (-2,-2.2) node[anchor=north] {$-k$};
			 \draw (3,-2) node[anchor=north] {$\nu(v)$};

		\draw[thick,-Triangle] (-4,-2) -- (-4,2);
			\draw (-3.9,0) -- (-4.2,0) node[anchor=east] {0};
			\draw (-4,2) node[anchor=east] {$\nu(u)$};
			
		\foreach \x in {0,.5,...,2.5}
			{\foreach \y in {0,.5,...,1.5}
				{\filldraw[fill=gray!40] (\x,\y) circle (6pt);}}
	
		\foreach \y in {0,-1,...,-3}
			{\foreach \x in {\y, ..., 5}
				{\draw[pattern={north west lines}] (0.5*\x,0.5*\y) circle (6pt);}}

		\foreach \x in {0,-1,...,-3}
			{\foreach \y in {\x,...,3}
				{\draw[pattern={north east lines}] (0.5*\x-2, 0.5*\y) circle (6pt);}}

		\foreach \x in {0,...,4}
			{\foreach \y in {0,...,3}
				{\draw[gray,ultra thick] (.5*\x-2, .5*\y) circle (5.6pt);
				\draw (.5*\x-2, .5*\y) circle (6pt);}}
	
		\foreach \x in {0,...,4}
			{\foreach \y in {0,...,3}
				{\draw (.5*\x-2-.5*\y, -.5*\y) circle (6pt);
				\filldraw[fill=gray] (.5*\x-2-.5*\y, -.5*\y) circle (2pt);}}

	\draw[thick, red] (-1.75,2) -- (-1.75,.25) -- (3,.25);
	\draw[rounded corners, thick, red] (-1.75,.5) -- (-1.75,-.25) -- (-2.25,-.25) -- (-2.25,-.75) -- (-2.75,-.75) -- (-2.75,-1.25) -- (-3.25,-1.25) -- (-3.25,-1.75) -- (-3.75,-1.75);

	\draw[rounded corners, thick, red, shift={(2,0)}] (-1.75,2) -- (-1.75,-.25) -- (-2.25,-.25) -- (-2.25,-.75) -- (-2.75,-.75) -- (-2.75,-1.25) -- (-3.25,-1.25) -- (-3.25,-1.75) -- (-3.75,-1.75);

	\draw (1.5,2.5) node[anchor=north] {$A_1$};
	\draw (3,-.75) node[anchor=west] {$A_2$};
	\draw (-2.75,2.5) node[anchor=north] {$A_4$};
	\draw (-.75,2.5) node[anchor=north] {$A_5$};	

    \end{tikzpicture}
	\begin{tikzpicture}
		\draw[pattern={north east lines}] (-3.5, -4.25) circle (6.5pt);
		\draw (-3.20, -4.25) node[anchor=west] {$|p^kv^2|_{p}$ maximal};
		\draw[pattern={north west lines}] (.5, -4.25) circle (6.5pt);
		\draw (.70, -4.25) node[anchor=west] {$|u^2|_{p}$ maximal};
		\filldraw[fill=gray!40] (4, -4.25) circle (6.5pt);
		\draw (4.20, -4.25) node[anchor=west] {$|1|_{p}$ maximal};
	\end{tikzpicture}
	\begin{tikzpicture}
		\draw (0,0) circle (6.5pt);
		\filldraw[fill=gray] (0,0) circle (2pt);
		\draw (.30, 0) node[anchor=west] {$|uv|_{p}$ maximal};
		\draw[gray,ultra thick] (4,0) circle (6.1pt);
		\draw (4,0) circle (6.5pt);
		\draw (4.30, 0) node[anchor=west] {$|v|_{p}$ maximal};
	\end{tikzpicture}
	\caption{A sketch of the partition of~$K^2$ used in the proof of \cref{prop:non-thety}, depicting the case~$k = 4$. Every circle represents a set of fixed valuation vector, i.e.\ of the form~$p^a\Op^\times \times p^b\Op^\times$ of~$K^2$.}\label{fig:divisions}
    \end{figure}
    
   \noindent Area~$A_1$, where~$|1|_{p}$ is maximal, is defined as~$p\Op^2$, hence
    \[
        \int_{A_1} |1|_{p}^{-1-s} \mathrm{d}\mu(w) = q^{-2}.
    \]
    Area~$A_2$ is defined as~$\bigcup_{j \in \N_0} p^{-j}\Op^\times \times p^{-j+1}\Op$, such that
    \begin{align*}
        \int_{A_2} |(u^2)(w)|_{p}^{-1-s} \mathrm{d}\mu(w) &= \sum_{j = 0}^{\infty} \int_{p^{-j}\Op^\times \times p^{-j+1}\Op} |u^2(w)|_{p}^{-1-s} \mathrm{d}\mu(w)\\
        &= q^{-1}\frac{(1-q^{-1})}{1-t^{2}}.
    \end{align*}
    We put \[A_3 = \bigcup_{j \in \N_0} p^{-j}(\Op^\times \times p(p^{-k}\Op \smallsetminus \Op)),\] and find
    \begin{align*}
        \int_{A_3} |uv(w)|_{p}^{-1-s} \mathrm{d}\mu(w) &=
		\sum_{j = 0}^\infty
			(1-q^{-1})t^j
			\sum_{l = 0}^{k-1} 
			\int_{p^{-j-l}\Op^\times} |v(w)|_{p}^{-1-s} \mathrm{d}\mu(v)
    \end{align*}
    by separation of the variables. The evaluation of the remaining sum yields
    \begin{align*}
        \sum_{l = 0}^{k-1} (1-q^{-1})t^{j+l}
        = (1-q^{-1})t^{j} \frac{1-t^{k}}{1-t}.
    \end{align*}
    Thus the integral over~$A_3$ evaluates to
    \begin{align*}
		(1-q^{-1})^2 \frac{1-t^k}{(1-t)(1-t^2)}.
    \end{align*}
    The fourth area is defined by~$A_4 = \bigcup_{j \in \N_0} p^{-j}(\Op \times p^{-k}\Op^\times)$, and we compute
    \begin{align*}
        \int_{A_4} |(p^kv^2)(w)|_{p}^{-1-s} \mathrm{d}\mu(w)
        &=\frac{(1-q^{-1})t^k}{1-t^{2}}.
    \end{align*}
    Finally, the last area~$A_5 = p(\Op \times (p^{-k}\Op \smallsetminus \Op))$ contributes with
    \begin{align*}
        \int_{A_5} |(v)(w)|_{p}^{-1-s} \mathrm{d}\mu(w) = q^{-1} \sum_{l = 0}^{k-1} \int_{p^{-l}\Op^\times} |(v)(w)|_{p}^{-1-s} \mathrm{d}\mu(w)
        &= q^{-1}(1-q^{-1})\, \frac{1-t^{k}}{1-t}.
    \end{align*}
    The sum of the five integrals calculated above factorises as
    \[
        \zeta_{S_k}^{H_k}(s) = q^{k}\frac{(q-t)(q(1-t^{k+1}) + t^2(1-t^{k-1})}{(1-t)^2(1+t)},
    \]
    and overall, we have
    \begin{align*}
        \zeta_{H_k}(s) &= q^{2k+2}\frac{(1-t)((1-(qt)^{k+1}) + qt^2(1-(qt)^{k-1})}{(1-qt)^2(1+qt)}
		\cdot \zeta_{\SL_2^1(\Op)}(s).
    \end{align*}
	Note that, depending on $k \bmod 2$, the numerator is divisible by $(1-qt)$ and $(1-(qt)^2)$, respectively. For $k = 0$, we recover the relative representation zeta function computed as the case $n=1$ in \cref{thm:diagonal on natural module}.
	
    Thus, the subgroup~$H_k$ of~$G$ is not thetyspectral. Indeed, the quotient of the respective zeta functions is
    \[
        \frac{\zeta_{H_k}(s)}{\zeta_{G}(s)} = q^{2k} 
		\frac{((1-(qt)^{k+1}) + qt^2(1-(qt)^{k-1})}
		{(1+t)(1-qt)}.
    \]
	Again, the numerator is in fact divisible by $(1-qt)$.
\end{proof}


\bibliographystyle{plain}
\bibliography{\jobname}

\end{document}